\documentclass{amsart}
\usepackage{graphicx}
\usepackage{xcolor,tikz}
\usepackage{amsthm,amssymb,mathrsfs,mathtools}
\usepackage{hyperref}
\usepackage{cleveref}
\usepackage{bbm}
\hypersetup{
    colorlinks,
    linkcolor={red!50!black},
    citecolor={blue!50!black},
    urlcolor={blue!80!black}
}

\usepackage{todonotes}

\newtheorem{theorem}{Theorem}[section]
\newtheorem*{theoremA}{Theorem}
\newtheorem{theoremM}{Theorem}

\newtheorem{lemma}[theorem]{Lemma}
\newtheorem{proposition}[theorem]{Proposition}
\newtheorem{corollary}[theorem]{Corollary}

\theoremstyle{remark}
\newtheorem{remark}[theorem]{Remark}
\theoremstyle{definition}
\newtheorem{definition}[theorem]{Definition}

\newtheorem{example}[theorem]{Example}
\newtheorem{notation}[theorem]{Notation}

\newtheorem*{disclaimer}{Disclaimer}

\def\orp{\mathscr{O}}

\newcommand{\eqehr}[3]{
\operatorname{EE}_{#1} ( {#2}, {#3} )
}
\newcommand{\ehr}[2]{
\operatorname{Ehr}_{#1} ( {#2} )
}
\newcommand{\hilb}[2]{
\operatorname{Hilb}({#1}, {#2} )
}
\newcommand{\eqhilb}[3]{
\operatorname{EHilb}_{#1}( #2 , #3 )
}

\def\Lip{\operatorname{Lip}}

\def\homchar{\mathfrak{b}}

\newcommand{\supp}[1]{\vert #1 \vert}
\newcommand{\usc}[1]{\underline{ #1 }}
\newcommand{\gl}[2]{\operatorname{GL}_{#1}(#2)}
\newcommand{\bij}[1]{\operatorname{Bij}( #1 )}

\newcommand{\basering}[1]{\operatorname{R}_{#1}}
\newcommand{\classf}[2]{\operatorname{C}_{#2}(#1)}
\newcommand{\conv}[1]{\operatorname{conv} #1 }
\newcommand{\Aut}[1]{\operatorname{Aut} #1 }
\newcommand{\FF}[1]{\mathcal F #1 }
\newcommand{\field}[1]{\mathbb K #1 }
\newcommand{\sgn}{\operatorname{sgn}}

\renewcommand{\AA}{\mathscr{A}}
\newcommand{\TT}{\mathscr{T}}
\newcommand{\fchar}{\mathfrak{f}}
\newcommand{\hchar}{\mathfrak{h}}
\newcommand{\mchar}{\mathfrak{Mon}}
\newcommand{\Oc}{\mathcal{O}}

\def\tW{\widetilde{W}}
\def\AW{\widetilde{\AA}_W}
\newcommand{\AWA}[1]{\widetilde{\AA}_{A_{#1}}}

\title[Equivariant Hilbert and Ehrhart series under translative actions]{Equivariant Hilbert and Ehrhart series\\under translative group actions}
\author{Alessio D'Al\`i}
\address{Dipartimento di Matematica, Politecnico di Milano, Italy}
\email{alessio.dali@polimi.it}
\author{Emanuele Delucchi}
\address{IDSIA-DTI, University of applied arts and sciences of Southern Switzerland, Lugano, Switzerland}
\email{emanuele.delucchi@supsi.ch}

\newcommand{\defil}[1]{{\em #1}}

\begin{document}

\begin{abstract}
We study representations of finite groups on Stanley--Reisner rings of simplicial complexes and on lattice points in lattice polytopes. The framework of translative group actions allows us to use the theory of proper colorings of simplicial complexes without requiring an explicit coloring to be given.

We prove that the equivariant Hilbert series of a Cohen--Macaulay simplicial complex under a translative group action 
admits a rational expression whose numerator is a positive integer combination of irreducible characters. This implies an analogous rational expression for the equivariant Ehrhart series of a lattice polytope with a unimodular triangulation that is invariant under a translative group action.

As an application, we study the equivariant Ehrhart series of alcoved polytopes in the sense of Lam and Postnikov and derive explicit results in the case of order polytopes and of Lipschitz poset polytopes.
\end{abstract}

\maketitle

\setcounter{tocdepth}{1}

\section{Introduction}

The {\em Ehrhart series} of a polytope $P$ whose vertices are in a lattice $M\subseteq \mathbb R^d$ is the formal power series
\begin{equation} \label{eq:non-equivariant Ehrhart}
\ehr{P}{t} = 1 + \sum_{m\geq 1} \vert mP \cap M \vert \cdot t^m
\end{equation}
where the degree $m$ coefficient counts the lattice points in the $m$-th dilation of $P$.

The {\em Hilbert series} of a finite nonempty simplicial complex $\Sigma$ is the Hilbert series of the (complex)\footnote{All statements about the Hilbert series of a Stanley--Reisner ring work over an arbitrary field $\mathbb{K}$; we choose $\mathbb{K} = \mathbb{C}$ for representation-theoretic reasons.} Stanley--Reisner ring of $\Delta$, i.e., of the quotient 
$$\mathbb C[x_v\mid v\in V(\Sigma)]/J_\Sigma$$
of the complex polynomial ring with variables indexed by vertices of $\Sigma$, where $J_\Sigma$ is the ideal generated by all monomials $x_{v_1}\cdots x_{v_k}$ such that $\{v_1,\ldots,v_k\}$ is not the set of vertices of a simplex in $\Sigma$. As an immediate consequence of the definition, the Stanley--Reisner ring $\mathbb{C}[\Sigma]$ admits a $\mathbb{C}$-basis consisting of those monomials supported on simplices of $\Sigma$, whence the alternative name ``face ring'' for $\mathbb{C}[\Sigma]$. One then has that

\begin{equation} \label{eq:non-equivariant Hilbert}
\mathrm{Hilb}(\mathbb{C}[\Sigma], t) = 1 + \sum_{i \geq 1}\dim_{\mathbb{C}}\mathbb{C}[\Sigma]_i \cdot t^i = 1 + \sum_{i \geq 1}|\mathrm{Mon}^i(\mathbb{C}[\Sigma])| \cdot t^i,
\end{equation}

\noindent where $\mathrm{Mon}^i(\mathbb{C}[\Sigma])$ denotes the set of monomials of $\mathbb{C}[\Sigma]$ of degree $i$.

The Ehrhart series of a lattice polytope and the Hilbert series of any of its unimodular lattice triangulations (if one exists) are linked as follows.

\begin{theoremA}[Betke--McMullen {\cite{BetkeMcMullen}}] Let $P$ be a lattice polytope and let $\Sigma$ be a unimodular lattice triangulation of $P$. Then
$$
\ehr{P}{t} = \hilb{\mathbb{C}[\Sigma]}{t}.
$$
\end{theoremA}

If a group $G$ acts on the given lattice polytope, respectively simplicial complex, equivariant analogues of the generating functions in \eqref{eq:non-equivariant Ehrhart} and \eqref{eq:non-equivariant Hilbert} are defined by replacing their coefficients by effective elements of the ring $\basering{G}$ of (complex) virtual characters of the group, see \Cref{sec:representation preliminaries}, \Cref{def:eqehr} and \Cref{subsec:eqHilb in general}.

This point of view led Stapledon to lay the foundations of \defil{equivariant Ehrhart theory} in \cite{Stapledon}. Since then, this topic has been investigated by several authors, see \cite{ASuV}, \cite{AScV}, \cite{CHK}. In particular, Elia, Kim and Supina \cite{EKS} developed methods for computing equivariant Ehrhart series via special subdivisions of the given polytope. 
The object of our paper is an equivariant version of the above-mentioned Betke-McMullen theorem (see the Disclaimer below) and its implications for the structure of the equivariant analogues of Ehrhart and Hilbert series. Our main result in this sense is \Cref{TMB} below, which we obtain as a consequence of a more general statement about group actions on simplicial complexes.

\medskip

The study of group actions on posets and simplicial complexes is a classical topic, see for instance work of Stanley \cite{StanleyActions}, Garsia and Stanton \cite{GarsiaStanton}. In particular, Stembridge \cite{StembridgeWeyl} introduced the notion of {\em proper} group actions on simplicial complexes.
More recently, Adams and Reiner \cite{AdamsReiner} have studied the equivariant Hilbert series of an abstract simplicial complex, focusing in particular on the case where the group action preserves a proper coloring (see \Cref{df:propercoloring}) of the complex itself. 

An action of a group $G$ on a partially ordered set is called {\em translative} if it satisfies a simple local condition (see \Cref{def:translative}). 
The word ``translative'' was introduced in \cite{DelucchiRiedel} in the context of hyperplane arrangements that are preserved by a group of translations, as an abstraction of the induced action on the arrangement's poset of intersections. 
Such group actions have appeared in various contexts in the literature, see \Cref{rem:translative}.  In the case of posets of faces of a simplicial complex, we observe that translative actions coincide with actions that preserve some proper coloring of the complex (\Cref{lem:translative action on sc}). Moreover, translativity is a strictly stronger condition than properness in the sense of Stembridge (\Cref{lem:translative implies proper} and \Cref{translativeisstronger}). Our first main result is the following characterization of translativity in terms of linear systems of parameters.

\begin{theoremM}[{see \Cref{prop:translativity via lsop} below}]\label{TMA}
Let $\mathbb{K}$ be an infinite field. The action of a group $G$ on a finite nonempty simplicial complex $\Sigma$ is translative if and only if the ring $\field[\Sigma]$ has a linear system of parameters $\{\theta_i\}_i$ where each $\theta_i$ lies in the invariant subring $\field[\Sigma]^G$.
\end{theoremM}

As a consequence, when $\Sigma$ is a Cohen--Macaulay complex and $G$ acts translatively on it, the equivariant Hilbert series of $\Sigma$ has a rational expression whose numerator is a polynomial with \emph{effective} coefficients (\Cref{thm:effective numerator for SR}). An equivariant version of the aforementioned Betke--McMullen theorem enables us to translate these considerations into the context of Ehrhart theory, obtaining the following result (see \Cref{subsec:eqHilb_SC_CM} and \Cref{subsec:eqHilb_colorful} for the definition of $\hchar_i$ and $\hchar_S$).

\begin{theoremM}[{see {\Cref{thm:EBM}} below}]\label{TMB}
Let $M \subseteq \mathbb{R}^d$ be a lattice and let $P \subseteq \mathbb{R}^d$ be a $d$-dimensional $M$-lattice polytope with a unimodular lattice triangulation $\Delta$. Let $\rho\colon G \to \mathrm{Aff}(M)$ be an action of a finite group $G$ by affine transformations of $M$ preserving $\Delta$. Then the following equality holds in $\basering{G}[[t]]$:
\begin{equation} \label{eq:intro_ehrhart_hilbert}
\eqehr{\rho}{P}{t} = \eqhilb{\rho}{\mathbb{C}[\Delta]}{t}.
\end{equation}
Assume further that the action $\rho$ is translative on $\Delta$ and let $\gamma\colon V(\Delta) \to \Gamma$ be a proper coloring of $\Delta$ preserved by $\rho$. Then 
\begin{equation} \label{eq:intro_translative_ehrhart}
    \eqehr{\rho}{P}{t} = \frac{\sum_{S \subseteq \Gamma} \hchar^{\Delta}_{S} t^{|S|}}{(1-t)^{|\Gamma|}} = \frac{\sum_{i=0}^{d+1}\hchar_i^{\Delta}t^i}{(1-t)^{d+1}},
\end{equation}
where each $\hchar^{\Delta}_i$ (but not necessarily each $\hchar^{\Delta}_S$) is effective. If moreover $|\Gamma|=d+1$, i.e.,~if $\gamma$ is balanced, then each $\hchar^{\Delta}_S$ in \eqref{eq:intro_translative_ehrhart} is effective.
\end{theoremM}

\medskip

Even in its most restrictive form, \Cref{thm:EBM} applies to some natural and important examples of polytopes and affords explicit computations. 
A notable example is given by {\em order polytopes} of posets and, more generally, {\em alcoved polytopes}.

\medskip

The {\emph{order polytope}} $\orp(X)$ associated with a partially ordered set $X$ has been defined by Stanley \cite{StanleyTwo} and has received wide attention ever since. 
We show that every poset automorphism of $X$ induces a translative action on the canonical regular unimodular triangulation of $\orp(X)$ (see \Cref{sec:OP}), and such an action preserves a balanced coloring of the triangulation (\Cref{lem:translative_order_polytope}). We give a formula for the equivariant Ehrhart series of $\orp(X)$ 
and we illustrate our result by some explicit computations (see \Cref{ex:radiotower} and \Cref{prop:joinorder}).

\begin{theoremM} [{see {\Cref{thm:orderp}} below for a more detailed statement}] \label{TMC}
Let $G$ be a group of automorphisms of a finite poset $X$. Then the equivariant Ehrhart series of $\orp(X)$ coincides with the equivariant Hilbert series of its canonical triangulation. Moreover, all entries of the equivariant flag $h$-vector are effective.
\end{theoremM}

\medskip

{\em Alcoved polytopes} have been introduced by Lam and Postnikov \cite{LaPo} as $\mathbb Z^d$-lattice polytopes that are triangulated by a certain standard simplicial subdivision of $\mathbb R^d$ (see \Cref{eq:Ad}). The same authors in \cite{LaPo2} later extended the theory to polytopes that are triangulated by the ``alcoves'' of any affine Coxeter system (see \Cref{def:Walcoved} for a precise definition), with the previous case corresponding to the Coxeter case $A_d$ via an affine isomorphism that we describe in \Cref{twoalcoves}. In \Cref{sec:alcoved} we show that, if $P$ is an alcoved polytope that is invariant by the action of some subgroup of the associated affine Coxeter group, then the induced action on the alcoved triangulation of $P$ satisfies the most restrictive hypotheses of \Cref{thm:EBM} and thus, for instance, has effective equivariant $h$- and flag $h$-vectors.

\begin{theoremM} [{see {\Cref{thm:Walcove}} below for a more detailed statement}] \label{TMD}
Let $P$ be an alcoved polytope of type $W$ and let $G$ be a subgroup of $\tW$ preserving $P$. Then the equivariant Ehrhart series of $P$ equals the equivariant Hilbert series of the alcoved triangulation of $P$ (with respect to the induced $G$-action). Morever, all entries of the equivariant flag $h$-vector are effective.
\end{theoremM}

Even though order polytopes are a special case of alcoved polytopes (see \Cref{rem:ORPA}), we decided to treat the former separately for the benefit of those readers not familiar with the Coxeter viewpoint.

Another special class of alcoved polytopes associated with posets is that of {\emph{Lipschitz polytopes of posets}} in the sense of Sanyal and Stump \cite{SaSt}. Again, every automorphism of the poset $X$ induces a translative action on the triangulation of the Lipschitz polytope of $X$ by alcoves, preserving a natural balanced coloring. In \Cref{ex:lips} we exemplify how the equivariant Ehrhart series of a poset Lipschitz polytope  can be computed via \Cref{thm:EBM}, verifying the effectiveness of the associated equivariant $h$- and flag $h$-vectors. 

\medskip

Our paper is organized as follows. In \Cref{sec:lpt} we lay out our setup on lattice polytopes and triangulations, we define translative and proper group actions and relate them to other notions of group actions in the literature. \Cref{sec:eHS} introduces the concept of equivariant Hilbert series, while \Cref{sec:eHS_SR} specializes it to the context of Stanley--Reisner rings, proving \Cref{TMA} and its consequences for the case of Cohen--Macaulay complexes. In \Cref{sec:main} we recall some basics of equivariant Ehrhart theory and relate the equivariant Ehrhart series of a lattice polytope and the equivariant Hilbert series of a unimodular triangulation under suitable hypotheses, in particular proving \Cref{TMB}. \Cref{sec:OP} is devoted to the study of the equivariant Ehrhart series of order polytopes under an action of a finite group via poset automorphisms and contains \Cref{TMC} among other results. Finally, in \Cref{sec:alcoved} we apply our results to alcoved polytopes, proving in particular \Cref{TMD}, and carry out some explicit computations in the case of Lipschitz polytopes of posets.

\medskip

\begin{disclaimer}
 An earlier version of this manuscript proved that \eqref{eq:intro_ehrhart_hilbert} holds under some additional hypothesis on the action $\rho$, see \Cref{rem:alternative proof of eBM}. We subsequently learned that the more general version stated above was proved in unpublished notes by Katharina Jochemko, Lukas Katth\"an and Victor Reiner. We thank them for sharing their unpublished work and allowing us to include their proof.

  Moreover, as the present paper was in preparation, a preprint by Alan Stapledon \cite{StapledonNew} appeared on arXiv. After some friendly exchanges, the updated versions of both our paper and Stapledon's make use of each other's results. In particular, Stapledon obtained an even more general form of \eqref{eq:intro_ehrhart_hilbert}, see \cite[Proposition 4.40, Remark 4.41]{StapledonNew} and  \Cref{rem:disclaimer} below.

\end{disclaimer}

\noindent {\bf Acknowledgements.} We thank Alan Stapledon for friendly discussions on his latest work \cite{StapledonNew} and catching some mistakes in an earlier draft, Victor Reiner for sharing his unpublished work with Katharina Jochemko and Lukas Katth\"an, the anonymous referee and Matt Beck for their valuable comments. We also thank Mariel Supina for introducing us to the topic with her talk at the workshop {\em Combinatorial and Algebraic Aspects on Lattice Polytopes} at Kwansei Gakuin University, and we are grateful to the organizers of this workshop for their support. AD is a member of INdAM--GNSAGA and has been partially supported by the PRIN 2020355B8Y grant ``Squarefree Gr\"obner degenerations, special varieties and related topics'' and by the INdAM--GNSAGA project ``New theoretical perspectives via Gr\"obner bases'', CUP E53C22001930001.

\newpage
\tableofcontents

\section{Lattice polytopes, simplicial complexes and group actions}\label{sec:lpt}

Throughout this paper we fix an integer $d>0$ and let $M$ denote a rank $d$ discrete subgroup of $\mathbb R^d$.

\subsection{Definitions}
A convex polytope $P$ is the convex hull of finitely many points in $\mathbb R^d$. The set of vertices of a polytope $P$ will be denoted by $V(P)$. If $V(P)\subseteq M$, we call $P$ a \defil{lattice polytope}. A geometric simplex is the convex hull of any set of affinely independent points. A lattice simplex $S=\operatorname{conv}\{x_0,x_1,\ldots,x_k\}$ is called \defil{unimodular} if the vectors $x_1-x_0,\ldots,x_k-x_0$ form a lattice basis of $\{\sum_{i=0}^k \lambda_i x_i \mid \lambda_0+\ldots +\lambda_k = 1\} \cap M$ (i.e., of the subset of $M$ contained in the affine span of $S$) \cite[\S 10.1]{BeckRobins}.

A \defil{geometric simplicial complex} in $\mathbb R^d$ is a set $\Delta$ of geometric simplices with the properties that $\Delta$ contains every face (including the empty face) of any of its elements, and that the intersection of any two simplices in $\Delta$ is a face of both \cite[\S 12.1]{TopologicalMethods}. The family $\Delta$ ordered by inclusion is a partially ordered set denoted by $\FF(\Delta)$ and called the \emph{poset of faces} of $\Delta$. Write $V(\Delta)=\bigcup_{\sigma\in \Delta}V(\sigma)$ for the set of vertices of $\Delta$.  
The support of a geometric simplicial complex $\Delta$ is the space $\supp{\Delta}:=\bigcup_{\sigma\in \Delta} \sigma \subseteq \mathbb R^d$. 
A geometric simplicial complex $\Delta$ is a {\em triangulation} of the polytope $P$ if $\supp{\Delta}=P$. A \defil{lattice triangulation} of a lattice polytope $P$ is any triangulation of $P$ whose vertices are points of $M$. A lattice triangulation $\Delta$ is \defil{unimodular} if every $\sigma\in \Delta$ is unimodular.

An \defil{abstract simplicial complex} $\Sigma$ on a finite set $W$ is any family of subsets of $W$ that is closed under taking subsets. A {\em subcomplex} of $\Sigma$ is any abstract simplicial complex $\Sigma'$ such that $\Sigma'\subseteq\Sigma$. The vertex set of $\Sigma$ is 
$V(\Sigma):=\bigcup_{\sigma\in \Sigma} \sigma
\subseteq W
$ and the poset of faces $\FF(\Sigma)$ is the  family $\Sigma$ itself, partially ordered by inclusion. The dimension of $\sigma\in \Sigma$ is $\dim(\sigma):=\vert \sigma\vert -1$, and the dimension of $\Sigma$ is the maximum of the dimensions of its elements. 
With every geometric simplicial complex $\Delta$ is associated an underlying abstract simplicial complex
$
\usc{\Delta} := \{V(\sigma) \mid \sigma\in \Delta\}
$. Note that $V(\Delta)=V(\usc{\Delta})$.

We say that an abstract simplicial complex $\Sigma$ is \emph{Cohen--Macaulay} over the field $\mathbb{K}$ if $\widetilde{H}_i(\mathrm{lk}_{\Sigma}(F), \mathbb{K}) = 0$ for every $F \in \Sigma$ and for every $i < \dim(\mathrm{lk}_{\Sigma}(F))$. Here $\mathrm{lk}_{\Sigma}(F)$ is the \emph{link} of $F$ in $\Sigma$, i.e.~the subcomplex of $\Sigma$ defined as $\{G \in \Sigma \mid F \cap G = \emptyset, \ F \cup G \in \Sigma\}$. If $\Sigma$ is Cohen--Macaulay over $\mathbb{K}$, then it is pure, i.e.~all facets of $\Sigma$ have the same cardinality.

\begin{definition} \label{df:propercoloring}
A \emph{proper coloring} of an abstract simplicial complex $\Sigma$ with values in the set of colors $\Gamma$ is any function $\gamma: V(\Sigma)\to\Gamma$ with the property that $\gamma(x)\neq \gamma(y)$ for every simplex $\sigma\in \Sigma$ and every pair of distinct elements $x,y\in \sigma$. If there exists a proper coloring with as few colors as possible, i.e.~$|\Gamma| = \dim(\Sigma)+1$, the complex $\Sigma$ is said to be \emph{balanced}.
\end{definition}

\subsection{Group actions}

An {\em affine transformation} of a $d$-dimensional $\mathbb{K}$-vector space $V$ is any function $f\colon V\to V$ of the form $f(x)= Ax + b$ with $A\in \mathrm{GL}(V)$ and $b\in V$. The set $\operatorname{Aff}(V)$ is a group under composition.

An action of a group $G$ on $\mathbb R^d$ by {affine transformations} is a group homomorphism $\rho:G\to\operatorname{Aff}(\mathbb R^d)$.
We say that the action $\rho$ preserves a geometric simplicial complex $\Delta$ if $\rho(g)(\Delta)=\Delta$ for all $g\in G$. Given an action of $G$ on $\Delta$, for every $g\in G$ we define $$\Delta^g:=\{\sigma\in \Delta \mid \rho(g)(\sigma)=\sigma\},$$ the set of simplices of $\Delta$ that are (setwise) fixed by $g$.

Given a polytope $P$ in $\mathbb R^d$ and $g\in G$, let
$$
P^g:=\{x\in P \mid \rho(g)(x)=x\}
$$
denote the set of fixed points of $g$ contained in $P$.

An action of $G$ on an abstract simplicial complex $\Sigma$ is a group homomorphism $\rho: G \to \bij{V(\Sigma)}$ into the symmetric group on $V(\Sigma)$ such that $\rho(g)(\Sigma)=\Sigma$ for all $g\in G$. For $g\in G$ let
$$\Sigma^g:=\{\tau\in \Sigma \mid \rho(g)(\tau)=\tau\}.$$ Any group action on a geometric simplicial complex $\Delta$ induces an action on its underlying abstract simplicial complex $\usc{\Delta}$. In this case, we write $\usc{\Delta}^g$ for $(\usc{\Delta})^g = \usc{(\Delta^g)}$.

An action of $G$ on a partially ordered set $X$ is a homomorphism $G\to \Aut(X)$ from $G$ into the group of automorphisms of $X$. (Recall that a poset automorphism is an order-preserving self-map with an order-preserving inverse. For basics on posets, we refer the reader to \cite[Chapter 3]{StanleyEnumerative}.) For $x,y\in X$ we define the set $x\vee y$ of all minimal upper bounds of $x$ and $y$ and the set $x\wedge y$ of all minimal upper bounds of $x$ and $y$. If $x\wedge y =\{z\}$ is a singleton, we abuse notation writing simply $x\wedge y = z$, and similarly for $x\vee y$.

\subsection{Translative and proper actions}

We now need to introduce \emph{translative} and \emph{proper} group actions. The reader interested in color-preserving actions on simplicial complexes will find a strong connection to translativity in \Cref{lem:translative action on sc}.

\begin{definition}[Translative and proper actions] \label{def:translative}
An action of a group $G$ on a partially ordered set $X$ is called:
\begin{itemize}
    \item \emph{translative} if, for every $g\in G$ and every $x\in X$, $gx\vee x\neq\emptyset$ implies $gx=x$;
    \item \emph{proper} if, for every $g \in G$ and every $x \in X$, $gx = x$ implies that $gy = y$ for every $y \leq x$.
\end{itemize}

An action of $G$ on a simplicial complex $\Sigma$ is called translative (respectively, proper) if the induced action on $\FF(\Sigma)$ is translative (respectively, proper).
\end{definition}

\begin{remark}\label{rem:translative} Translative actions on posets arose from the study of matroid structures related to toric arrangements \cite{DelucchiRiedel}, and were further investigated in \cite{DaliDelucchi} (in the context of invariant rings of Stanley--Reisner rings) and in \cite{BibbyDelucchi} (in connection to supersolvability of geometric posets).
   If $X$ is the poset of faces of a simplicial fan, proper actions are Stembridge's proper actions \cite{StembridgeWeyl}. If $X$ is any poset regarded as a small category without loops, any group action in the sense of Bridson and Haefliger \cite[Definition 1.11]{BriHae} is proper.
   \end{remark}

The following lemma addresses the relationship between translative actions and color-preserving automorphisms of a simplicial complex in the sense of \Cref{df:propercoloring}.

\begin{lemma} \label{lem:translative action on sc}
Let $\Sigma$ be a simplicial complex and let $G$ be a group acting on it. Then the action is translative if and only if it preserves a proper coloring of $\Sigma$, i.e.,~there exists a proper coloring $\gamma$ such that $\gamma(gv) = \gamma(v)$ for every $v \in V(\Sigma)$ and $g \in G$.
\end{lemma}
\begin{proof}
First of all, note that both translativity of an action and properness of a coloring only depend on the structure of the face poset $\FF(\Sigma)$. Hence, it is irrelevant whether the simplicial complex $\Sigma$ is abstract or geometric.

Assume first that the action of $G$ is translative. Consider the partition of the vertices of $\Sigma$ induced by the orbits of the action of $G$, and assign a different color to each orbit. This coloring is obviously preserved by $G$, and we will now prove that it is proper. Indeed, if $\{v, w\}$ is a monochromatic face of $\Sigma$, it must be that $w = gv$ for some $g \in G$. But then, by translativity, it follows that $gv = v$, and hence $w = v$. Thus the only monochromatic faces are vertices, which proves properness of the coloring.

Now let $\gamma$ be a proper coloring of $\Sigma$ preserved by the action $\rho$. Pick a face $\sigma \in \Sigma$ and an element $g \in G$ such that $g\sigma \vee \sigma$ is nonempty in $\FF(\Sigma)$, i.e,~there exists a face of $\Sigma$ containing both $\sigma$ and $g\sigma$. If $\sigma$ is the empty face of $\Sigma$, there is nothing to prove; assume otherwise. Since $g\sigma \vee \sigma$ is nonempty, for every vertex $v \in \sigma$ one has that $\{v, gv\} \in \Sigma$. Since $\gamma(v) = \gamma(gv)$ and the coloring $\gamma$ is proper, the only possibility is that $gv = v$. Hence, $g\sigma = \sigma$ and the action is translative.
\end{proof}

\Cref{lem:translative action on sc} shows that translativity of an action on a simplicial complex only depends on the 1-skeleton of the complex itself. We now make this observation more precise.

\begin{remark} \label{rem:quotient coloring}
Let $G$ be a finite group acting via simplicial automorphisms on the finite simplicial complex $\Sigma$, let $V$ be the set of vertices of $\Sigma$, and let $\Sigma^{(1)}$ be the 1-skeleton of $\Sigma$ (considered as a graph). We can then consider the graph obtained as the topological quotient $\Sigma^{(1)}/G$, whose vertex set is the set $V/G$ of $G$-orbits of vertices of $\Sigma$, and with an edge with ends at the orbits $\Oc$, $\Oc'$  whenever there exist $v \in \Oc$ and $v' \in \Oc'$ such that $v \neq v'$ and $\{v, v'\} \in \Sigma$.   
In particular, $\Sigma^{(1)}/G$ has one loop at $\Oc$ for every pair of distinct $v$, $v'$ in $\Oc$ such that $\{v, v'\} \in \Sigma$. Proper colorings of $\Sigma^{(1)}/G$ are then in bijection with those proper colorings of $\Sigma$ that are preserved by the action of $G$. In particular, the action of $G$ on $\Sigma$ is translative if and only if the graph $\Sigma^{(1)}/G$ is loopless or, equivalently, colorable. When this is the case, the number of colors of a proper coloring of $\Sigma$ preserved by the $G$-action is bounded above by $|V/G|$ and bounded below by the chromatic number of the graph $\Sigma^{(1)}/G$.
\end{remark}

\begin{corollary}
Let $\Sigma$ be a finite simplicial complex, let $G$ be a finite group acting on it via simplicial automorphisms, and let $k$ be the number of $G$-orbits of vertices of $\Sigma$. If $k \leq \dim(\Sigma)$, then the action of $G$ on $\Sigma$ is not translative. 
\end{corollary}

   We now wish to investigate the relationship between translativity and properness, beginning with the following immediate lemma.

\begin{lemma} \label{lem:translative implies proper}
Let $X$ be a partially ordered set and let $G$ be a group acting on $X$. If the action of $G$ is translative, then it is proper.
\end{lemma}
\begin{proof}
Let $g \in G$ and $x \in X$ be such that $gx = x$. Take $y \in X$ such that $y \leq x$. Since $gy \leq gx = x$, the set $gy \vee y$ is nonempty; hence, by translativity, it follows that $gy = y$.
\end{proof}

\begin{remark}\label{translativeisstronger} 
Translativity is strictly stronger than properness; for an example of a proper action failing to be translative, take the ``butterfly'' poset $X$ with elements $x_0, x_1, y_0, y_1$ and covering relations $y_i < x_j$ for every $i,j \in \{0,1\}$. 
The action of $G = \mathbb{Z}_2 = \{1, g\}$ on $X$ defined by $g x_i = x_{1-i}$ and $g y_i = y_{1-i}$ for every $i \in \{0,1\}$ is trivially proper, but it is not translative since $g y_0 \vee y_0 = \{x_0, x_1\}$ is nonempty and $y_0 \neq g y_0 = y_1$.

However, if the group $G$ has exponent two (i.e.,~all the elements of the group $G$ have order at most two) and $X$ is a meet-semilattice (i.e.,~for every pair of distinct elements $x, y \in X$ the set of maximal lower bounds consists of a unique element, the \emph{meet}),
then properness and translativity coincide.
\end{remark}

\begin{lemma}
Let $G$ be a group of exponent two acting on a meet-semilattice $X$. Then the action of $G$ is translative if and only if it is proper.
\end{lemma}
\begin{proof}
By \Cref{lem:translative implies proper}, if the action is translative, then it is proper. Let us now assume that the action is proper. Let $x \in X$ and $g \in G$ be such that $gx \vee x$ is nonempty. We need to prove that $gx = x$. Let $y$ be an element of $X$ such that $x \leq y$ and $gx \leq y$. Since $G$ acts by automorphisms of $X$, it is also true that $g^2x \leq gy$, i.e.~$x \leq gy$, since $g^2 = \mathrm{id}_G$ because of the hypothesis on $G$. Hence, $x$ lies below both $y$ and $gy$, and hence below the meet $gy \wedge y$ (that exists since $X$ is a meet-semilattice). If we can prove that $g(gy \wedge y) = gy \wedge y$, then properness will imply that $x = gx$, as desired. Since $G$ acts by automorphisms of $X$ and $X$ is a meet-semilattice, then the automorphism associated with $g$ must preserve the meet operation. In particular, \[g(gy \wedge y) = g^2y \wedge gy = y \wedge gy = gy \wedge y,\]
which ends the proof.
\end{proof}

\begin{corollary}\label{cor:tranproper}
Let $G$ be a group of exponent two acting on the face poset of a polytopal complex or of a simplicial complex. Then the action of $G$ is translative if and only if it is proper.
\end{corollary}

We close this section by proving two properties of proper actions that will come in handy when dealing with equivariant Hilbert and Ehrhart series, see \Cref{prop:eqhilb_proper}.(ii) and \Cref{rem:alternative proof of eBM}. As a slogan, when an action on a simplicial complex $\Sigma$ is proper, we will be able to read off some information about $\Sigma$ by considering all the possible $g$-fixed sets $\Sigma^g$ (that in this case are in fact simplicial complexes) as $g$ ranges over $G$.

\begin{lemma} \label{lem:proper action on sc}
Let $\Delta$ be an abstract (respectively, geometric) simplicial complex and let $G$ be a group acting on it. Then the action is proper if and only if $\Delta^g$ is an abstract (respectively, geometric) simplicial complex for every $g \in G$.
\end{lemma}
\begin{proof}
If $\Delta$ is an abstract simplicial complex, then the claim is a rewording of the definition of proper action.

Now, let $\Delta$ be a geometric simplicial complex, and let $\usc{\Delta}$ be the underlying abstract simplicial complex. Since $\FF(\Delta) = \FF(\usc{\Delta})$, the action of $G$ on $\Delta$ is proper if and only if the action of $G$ on $\usc{\Delta}$ is. Moreover, given any $g \in G$, one has that $\Delta^g$ is a geometric simplicial complex if and only if $\Delta^g$ is a subcomplex of $\Delta$ or, equivalently, $\usc{\Delta}^g$ is a subcomplex of $\usc{\Delta}$ and this, in turn, is equivalent to $\usc{\Delta}^g$ being an abstract simplicial complex. The claim now follows from the first statement of this lemma.
\end{proof}

\begin{lemma} \label{lem:proper triangulation}
Let $P$ be an $M$-lattice polytope in $\mathbb R^d$ with a lattice triangulation $\Delta$, and let $\rho$ be an affine action of a group $G$ on $\mathbb R^d$ that preserves $M$ and $\Delta$. If the action of $G$ on $\Delta$ is proper, then for every $g \in G$ one has that $P^g$ is a lattice polytope triangulated by $\Delta^g$. If moreover $\Delta$ is unimodular, so is each $\Delta^g$.
\end{lemma}
\begin{proof}
    Assume the action of $G$ is proper and fix $g \in G$. By \Cref{lem:proper action on sc}, $\Delta^g$ is a geometric simplicial complex; moreover, all its vertices are lattice points of $P^g$, since they form a subset of the vertices of the original  lattice triangulation $\Delta$. If we can prove that \begin{equation}P^g = \bigcup_{\sigma \in \Delta^g}\sigma,\end{equation}
    then $\Delta^g$ will be a triangulation of $P^g$ and, as a consequence, $P^g$ will be a lattice polytope. 
Fix a face $\sigma = \conv\{v_1, \ldots, v_s\}$ of $\Delta$. If $\sigma \in \Delta^g$, then $\rho(g)(v_i) = v_i$ for every $i \in [s]$ by \Cref{lem:proper action on sc}. It then follows immediately that $\sigma \subseteq P^g$. This proves $P^g \supseteq \bigcup_{\sigma \in \Delta^g}\sigma$.

Let us now show that $P^g \subseteq \bigcup_{\sigma \in \Delta^g}\sigma$. Let $x \in P^g$; then, by definition, $x \in P$ and $\rho(g)(x) = x$. Since $\Delta$ triangulates $P$, the point $x$ lies in the relative interior of a unique face $\tau = \conv\{v_1, \ldots, v_s\}$ of $\Delta$  (see, e.g., \cite[\S 2.3.1]{DRS}). However, it is also true that $x = \rho(g)(x)$ lies in the relative interior of $\rho(g)(\tau) = \{\rho(g)(v_1), \ldots, \rho(g)(v_s)\}$, which is also a face of $\Delta$. By the aforementioned uniqueness, it follows that $\rho(g)(\tau) = \tau$, i.e.~$\tau \in \Delta^g$.

We have proved that $\Delta^g$ triangulates $P^g$. If $\Delta$ is unimodular then $\Delta^g$, being a subcomplex of $\Delta$, is unimodular as well (since faces of unimodular simplices are unimodular).
\end{proof}

\section{Background on equivariant Hilbert series} \label{sec:eHS}

\subsection{Representations of finite groups} \label{sec:representation preliminaries}

Let $G$ be a finite group and $\mathbb K$ be a field. A representation of $G$ over $\mathbb K$ is any homomorphism $\rho:G\to \gl{d}{\mathbb K}$. The representation ring $R(G, \mathbb K)$ (or ``Grothendieck ring of representations'') of $G$ over $\mathbb K$ is the ring generated by all isomorphism classes of representations of $G$ over $\mathbb K$ with addition and multiplication induced by taking direct sum, resp.\ tensor product of representations, see \cite[Definition 2.3]{AdamsReiner}.

With every representation $\rho$ is associated a character $\chi_\rho: G\mapsto \mathbb K$ defined by $\chi_\rho(g)=\operatorname{trace}(\rho(g))$. Every character is constant on conjugation classes of $G$ and thus is an element of the ring $\classf{G}{\mathbb K}$ of class functions. The ring of virtual characters $\basering{G}(\mathbb K)$ is the subring of $\classf{G}{\mathbb K}$ 
generated by all 
integer multiples of characters.  The function $\rho \mapsto \chi_\rho$ maps $R(G, \mathbb K)$ to $\basering{G}(\mathbb K)$. When $\mathbb K$ has characteristic zero this mapping is injective 
 and defines an isomorphism onto its image (see, e.g., \cite[Corollary 30.14]{CurtisReiner}). 
 
 In the following, when not otherwise stated we will consider complex representations and write
 $$ \basering{G}:=\basering{G}(\mathbb C)\cong R(G, \mathbb C).$$

If an element of $R(G,\mathbb{C})$  is the class of a representation of $G$ over $\mathbb{C}$, we will say that such an element (or, equivalently, its image in $\basering{G}$) is \emph{effective}. An element of $R(G,\mathbb{C})$ (respectively, $\basering{G}$) is effective if and only if it is a positive integer combination of irreducible $G$-representations (respectively, $G$-characters).

 Let $A(t)=\sum a_it^i$ be an element of the ring $\basering{G}[[t]]$ of formal power series over $\basering{G}$. Evaluating $A(t)$ at any $g\in G$ yields a formal power series
 $$
 A(t)(g) = \sum a_i(g) t^i \in \mathbb C[[t]].
 $$
\begin{remark} \label{rem:evaluation at every g}
   For $A(t), B(t)\in \basering{G}[[t]]$, we have $A(t)=B(t)$ if and only if $A(t)(g)=B(t)(g)$ for all $g\in G$. The ``only if'' implication is trivial. For the other statement write $A(t)=\sum a_it^i$, $B(t)=\sum b_it^i$. Fix an index $i$. If $a_i(g)=b_i(g)$ for every $g \in G$, then it must be that $a_i=b_i$, since characters are uniquely determined by their evaluations. Since this holds for every $i$ by assumption, we conclude that $A(t)=B(t)$ in $\basering{G}[[t]]$.
\end{remark}

\subsection{The equivariant Hilbert series of a Cohen--Macaulay standard graded algebra} \label{subsec:eqHilb in general}

Let $\mathbb{K}$ be a field and let $R = \bigoplus_{i \in \mathbb{N}}R_i$ be a standard graded $\mathbb{K}$-algebra, i.e.,~an $\mathbb{N}$-graded $\mathbb{K}$-algebra finitely generated by its degree one part $R_1$ and such that $R_0 = \mathbb{K}$.
Note that, letting $S := \mathrm{Sym}^{\bullet}_{\mathbb{K}}R_1$ be the symmetric algebra over $\mathbb{K}$ of the finite-dimensional $\mathbb{K}$-vector space $R_1$, one gets a natural surjection $S \twoheadrightarrow R$, whose kernel $I$ lives in degrees two and higher. In particular, $S_1$ and $R_1$ can be canonically identified.

Let $\rho\colon G \to \mathrm{GrAut}_{\mathbb{K}}(R)$ be an action of $G$ on $R$ by graded $\mathbb{K}$-algebra automorphisms. (We record here that any such $\rho$ lifts to an action $\hat{\rho}$ of $G$ on $S$ by graded $\mathbb{K}$-algebra automorphisms, and that $\hat{\rho}(g)(I) = I$ for every $g \in G$; in other words, $I$ is invariant under the $G$-action $\hat{\rho}$.)

Since by construction $G$ acts via $\rho$ on each graded component of $R$ (and acts trivially on $R_0 = \mathbb{K}$), it makes sense to consider the following formal power series in $R(G,\mathbb{K})[[t]]$:
\[\sum_{i \geq 0}[R_i]t^i = 1 + [R_1]t + [R_2]t^2 + \ldots,\] where $[R_i]$ is the isomorphism class of the $G$-representation on $R_i$ induced by $\rho$, as an equivariant analogue of the Hilbert series of $R$.
Since for the scope of this paper we are only interested in complex representations of finite groups, unless otherwise stated we will assume that $G$ is finite, $\mathbb{K} = \mathbb{C}$, and reserve the name \emph{equivariant Hilbert series} of $R$ for the following formal power series in $\basering{G}[[t]]$:
\begin{equation} \label{eq:equivariant Hilbert series}
\eqhilb{\rho}{R}{t} := \sum_{i \geq 0}\chi_{R_i}t^i = 1 + \chi_{R_1}t + \chi_{R_2}t^2 + \ldots,
\end{equation}
where $\chi_{R_i}$ is the character of the complex $G$-representation of $R_i$ induced by $\rho$.

In the non-equivariant case (over an arbitrary field $\mathbb{K})$, it is customary to express the Hilbert series rationally as $\frac{h(t)}{(1-t)^d}$, where $h(t)$ is a polynomial with integer coefficients and $d$ is the Krull dimension of $R$. Such a rational expression is unique and $h(t)$ is known as the \emph{$h$-polynomial} of $R$. However, when dealing with equivariant Hilbert series, things become more complicated. 
In the special case when $R$ is the Stanley--Reisner ring of a simplicial complex endowed with a proper action (in the sense of \Cref{def:translative}), we will see in \Cref{subsec:eqHilb_SC_CM} how to use a suitable multigraded structure of $R$ to obtain a rational expression of the equivariant Hilbert series with a power of $(1-t)$ as its denominator. Even in that case, though, other rational expressions are possible, and it is not a priori obvious which choice of a denominator is the most informative one.

Recall from commutative algebra that a \emph{homogeneous system of parameters} for the $\mathbb{N}$-graded $d$-dimensional $\mathbb{K}$-algebra $R$ is a collection $\theta_1, \ldots, \theta_d$ of homogeneous elements of $\mathfrak{m} = (x \in R \mid \deg(x)>0)$ such that $\sqrt{(\theta_1, \ldots, \theta_d)} = \mathfrak{m}$. If $\theta_1, \ldots, \theta_d$ all have the same degree $r$, we say that the homogeneous system of parameters has degree $r$; in the special case when $r=1$, we say that the linear forms $\theta_1, \ldots, \theta_d$ form a \emph{linear system of parameters} for $R$. If $z_1, \ldots, z_r \in R$ are such that $R/(z_1, \ldots, z_r) \neq 0$ and, for every $i$, the element $z_i$ is a nonzerodivisor on the quotient ring $R/(z_1, \ldots, z_{i-1})$, we say that the ordered sequence $z_1, \ldots, z_r$ is a \emph{regular sequence} for $R$. When the ring $R$ is Cohen--Macaulay, then every homogeneous system of parameters is a regular sequence, and this is in fact an alternative characterization of Cohen--Macaulayness for $R$: see, e.g., \cite[Theorem I.5.9]{StanleyCombCommAlg}.

For the rest of the section, we will assume that $G$ is finite, $R$ is Cohen--Macaulay, and $\mathbb{K} = \mathbb{C}$. Under these assumptions, the key observation is that \emph{different choices of $G$-invariant linear system of parameters will yield different rational expressions for the equivariant Hilbert series of $R$}. This was essentially already known to Stembridge \cite[proof of Theorem 1.4]{StembridgeWeyl} and has recently been carefully restated by Stapledon \cite[Lemma 4.8]{StapledonNew}. For ease of reference, we include here Stapledon's version of the statement, and refer the interested reader to \cite{StapledonNew} for a proof: 

\begin{lemma}[{\cite[Lemma 4.8]{StapledonNew}}] \label{lem:artinian_reduction}
Let $R = \bigoplus_{i \in \mathbb{N}}R_i$ be a positively graded, finitely generated, $d$-dimensional Cohen--Macaulay $\mathbb{C}$-algebra with $R_0 = \mathbb{C}$. Let $\rho$ be an action of a finite group $G$ on $R$ via graded $\mathbb{C}$-algebra automorphisms. Let ${\theta}_1, \ldots, {\theta}_d$ be a homogeneous system of parameters of degree $r$ for $R$ and let $J \subseteq R$ be the ideal generated by ${\theta}_1, \ldots, {\theta}_d$. Assume that $J$ is $G$-invariant. Then
\[\eqhilb{\rho}{R}{t} = \frac{\eqhilb{\overline{\rho}}{R/J}{t}}{\det(\mathrm{Id} - J_rt^r)},\]
where:
\begin{itemize}
\item $\overline{\rho}$ is the $G$-action induced by $\rho$ on the (Artinian) quotient $\mathbb{C}$-algebra $R/J$;   
\item $J_r$ is the degree $r$ component of $J$ (i.e., the $\mathbb{C}$-vector space generated by ${\theta}_1, \ldots, {\theta}_d$) with the representation induced by $\rho$;
\item $\det(\mathrm{Id} - J_rt^r) = \sum_{i=0}^{d}(-1)^i\chi_{\wedge^i\!J_r}t^i$.
\end{itemize}
\end{lemma}

\begin{corollary} \label{cor:effective numerator}
    Let $R = \bigoplus_{i \in \mathbb{N}}R_i$ be a standard graded Cohen--Macaulay $\mathbb{C}$-algebra. Let $\rho$ be an action of a finite group $G$ on $R$ via graded $\mathbb{C}$-algebra automorphisms. Let ${\theta}_1, \ldots, {\theta}_d$ be a linear system of parameters for $R$ and assume that each ${\theta}_i$ lies in the invariant ring $R^G$, i.e.,~$\rho(g)(\theta_i) = \theta_i$ for every $g \in G$. Then
\begin{equation} \label{eq:good rational eqHilb}
    \eqhilb{\rho}{R}{t} = \frac{\eqhilb{\overline{\rho}}{R/J}{t}}{(1-t)^d},
\end{equation}
where $J$ is the $G$-invariant ideal of $R$ generated by $\theta_1, \ldots, \theta_d$ and $\overline{\rho}$ is the $G$-action induced by $\rho$ on the (Artinian) quotient $\mathbb{C}$-algebra $R/J$. In particular, the numerator of the rational expression in \eqref{eq:good rational eqHilb} is effective.
\end{corollary}
\begin{proof}
This is a direct consequence of \Cref{lem:artinian_reduction}, after noting that $G$ acts trivially on $\bigwedge^i\!J_1$ for every $i$.
\end{proof}

\begin{example} \label{ex:equivariant Artinian reduction}
Let $R = \mathbb{C}[x,y]/(xy)$ and consider the $\mathbb{Z}_2$-action $\rho$ where the nontrivial element of $\mathbb{Z}_2$ swaps (the classes of) $x$ and $y$. Then, for each $i \geq 1$, one has that $\chi_{R_i} = 1 + \chi_{\sgn}$, where $\chi_{\sgn}$ is the character of the sign representation. Hence,
\[\eqhilb{\rho}{R}{t} = 1 + \sum_{i \geq 1}(1+\chi_{\sgn})t^i.\]
One checks that $R$ is Cohen--Macaulay and $1$-dimensional; hence, in order to apply \Cref{lem:artinian_reduction}, we just need to find a single regular element in $R_1$ generating a $\mathbb{Z}_2$-invariant ideal. The linear forms $\ell = x-y$ and $\ell' = x+y$ both have the desired properties. Let us first take $\ell = x-y$. In this case, the ideal $J = (\ell)$ is $\mathbb{Z}_2$-invariant, and since the nontrivial element of $\mathbb{Z}_2$ sends $\ell$ to $-\ell$ we get that $\det(I-J_1t) = 1 - \chi_{\sgn}t$. Since $\overline{x} \in R/J$ is fixed by the $\mathbb{Z}_2$-action, the equivariant Hilbert series of the Artinian ring $R/J$ is $1+t$. Overall, by \Cref{lem:artinian_reduction}, we get that \[\eqhilb{\rho}{R}{t} = \frac{1+t}{1-\chi_{\sgn}t}.\]
Now take $\ell' = x+y$ and $J' = (\ell')$. Now the denominator is easier to compute; indeed, since $\ell'$ is fixed by the $\mathbb{Z}_2$-action, we just get that $\det(I-J'_1t) = 1 - t$. Since the nontrivial element of $\mathbb{Z}_2$ sends $\overline{x} \in R/J'$ to $-\overline{x}$, the equivariant Hilbert series of $R/J'$ is $1+\chi_{\sgn}t$, and overall we have that 
\[\eqhilb{\rho}{R}{t} = \frac{1+\chi_{\sgn}t}{1-t}.\]
This example, despite its simplicity, is quite instructive: indeed, $\{\ell'\} = \{x+y\}$ is a linear system of parameters for $R$ that is contained in the invariant subring $R^{\mathbb{Z}_2}$. We will see in \Cref{prop:translativity via lsop} a characterization of when such linear systems of parameters can be constructed for Stanley--Reisner rings. On the other hand, choosing the linear system of parameters $\{\ell\} = \{x-y\}$ yields a palindromic numerator in the rational expression of the equivariant Hilbert series of $R$, mirroring what happens for the non-equivariant Hilbert series of the Gorenstein ring $R$. 
\end{example}

\section{Equivariant Hilbert series under translative actions} \label{sec:eHS_SR}

\subsection{The Cohen--Macaulay approach} \label{subsec:eqHilb_SC_CM}

The goal of this section is to show that, when $\Sigma$ is a simplicial complex that is Cohen--Macaulay over $\mathbb{C}$ and $G$ is a finite group acting translatively on it, the Stanley--Reisner ring $\mathbb{C}[\Sigma]$ fits the setting of \Cref{cor:effective numerator}. 

Let us first consider a more general setup where $\Sigma$ is a $d$-dimensional simplicial complex and $\rho$ is a proper (in the sense of \Cref{def:translative}) action of a finite group $G$ on $\Sigma$. Let us denote by $\Sigma/G$ the set of $G$-orbits of faces of $\Sigma$. We can make $\Sigma/G$ into a poset by defining the following partial order: $\Oc' \preceq_{\Sigma/G} \Oc$ if, given any representatives $\tau$ and $\sigma$ for $\Oc'$ and $\Oc$ respectively, there exists $g\in G$ such that $g\tau \subseteq \sigma$. The poset $\Sigma/G$ is graded if we define its rank function to be the one inherited from (the face poset of) $\Sigma$: for any $\Oc \in \Sigma/G$, we set $\mathrm{rk}_{\Sigma/G}(\Oc) = |\sigma| = \dim(\sigma)+1$, where $\sigma$ is any representative of $\Oc$.

Let $\Oc \in \Sigma/G$. Each $\rho(g)$ permutes the faces of $\Sigma$ belonging to the orbit $\Oc$, and one can thus define the character $\fchar_{\Oc}^{\Sigma}$ of the associated complex permutation representation. (When $\Oc$ is the orbit of the empty face of $\Sigma$, we set $\fchar_{\Oc}^{\Sigma} = 1$.) Defining $\fchar_{i-1}^{\Sigma}$ to be the character of the complex permutation representation on the set of all $(i-1)$-dimensional faces of $\Sigma$, it follows that $\fchar_{i-1}^{\Sigma} = \sum_{\mathrm{rk}(\Oc)=i}\fchar^{\Sigma}_{\Oc}$.

For every $k \in \{0,1,\ldots,d+1\}$, let us denote by $\hchar_k^{\Sigma}$ the virtual character
\[\hchar_k^{\Sigma} := \sum_{i=0}^{k}(-1)^{k-i}\binom{d+1-i}{k-i}\fchar_{i-1}^{\Sigma}.\]

The collections $(\fchar_{i-1}^{\Sigma})_{i=0}^{d+1}$ and $(\hchar_{i}^{\Sigma})_{i=0}^{d+1}$ are the equivariant versions of the $f$-vector and the $h$-vector of $\Sigma$, and yield precisely those when evaluated at the identity element of $G$.

We now collect for later use some useful results already known to Stembridge \cite[Lemmas 1.1 and 1.2]{StembridgeWeyl}. Here and in what follows, if $m = \prod_{i \in V(\Sigma)}x_i^{\alpha_i}$ is a monomial in $\mathbb{C}[\Sigma]$, we will call \emph{support} of $m$ the set $\mathrm{supp}(m) := \{i \in V(\Sigma) \mid \alpha_i > 0\}$. It follows immediately from the definition of Stanley--Reisner ring that the support of any monomial in $\mathbb{C}[\Sigma]$ must be a face of $\Sigma$.

\begin{proposition} \label{prop:eqhilb_proper}
    Let $\Sigma$ be a finite abstract simplicial complex and let $\rho$ be a proper action of the finite group $G$ on $\Sigma$. Then,

    \begin{enumerate}
    \item the following equalities hold in $\basering{G}[[t]]$:
    \begin{equation} \label{eq:eqhilb_proper}
    \eqhilb{\rho}{\mathbb{C}[\Sigma]}{t} = 
    \sum_{\Oc \in \Sigma/G} \fchar^{\Sigma}_{\Oc}
\frac{t^{\mathrm{rk}(\Oc)}}{(1-t)^{\mathrm{rk}(\Oc)}},
\end{equation}
    
    \begin{equation} \label{eq:hilbfaces_general}
\eqhilb{\rho}{\mathbb{C}[\Sigma]}{t} = \sum_{i=0}^{d+1} \fchar^{\Sigma}_{i-1}
\frac{t^{i}}{(1-t)^{i}} 
= \frac{\sum_{i=0}^{d+1} \hchar^\Sigma_i t^i}{(1-t)^{d+1}},
\end{equation}
where $d=\dim(\Sigma)$.

    \item for every $g \in G$, the set $\Sigma^g$ is a simplicial complex and the Hilbert series of $\mathbb{C}[\Sigma^g]$ coincides with the evaluation at $g$ of $\eqhilb{\rho}{\mathbb{C}[\Sigma]}{t}$.
    \end{enumerate}

\end{proposition}
\begin{proof}
\phantom{}
\begin{enumerate}
\item The first equality in \eqref{eq:hilbfaces_general} comes from combining \eqref{eq:eqhilb_proper} with the fact that $\fchar_{i-1}^{\Sigma} = \sum_{\mathrm{rk}(\Oc)=i}\fchar^{\Sigma}_{\Oc}$. The second equality in \eqref{eq:hilbfaces_general} is proved by a straightforward computation akin to the non-equivariant case: see, e.g., \cite[Chapter II]{StanleyCombCommAlg}.

Let us now prove that \eqref{eq:eqhilb_proper} holds. This is essentially \cite[Lemma 1.1]{StembridgeWeyl}, and we will reprove it here for completeness. Let $m$ be a monomial of $\mathbb{C}[\Sigma]$ supported on $\sigma \in \Sigma$. Since the action of $\rho$ is proper by assumption, one has that $gm = m$ if and only if $g\sigma = \sigma$, i.e.,~the supports of $gm$ and $m$ coincide. Indeed, if $g\sigma = \sigma$, then properness dictates that $gv = v$ for every vertex $v \in \sigma$, and hence $gm = m$. Vice versa, if $gm = m$, then $gm$ and $m$ must be supported on the same face of $\Sigma$, and thus $g\sigma = \sigma$.

Now let $\Oc \in \Sigma/G$ be the $G$-orbit of some nonempty face in $\Sigma$. The action $\rho$ permutes in a degree-preserving fashion the set of monomials of $\mathbb{C}[\Sigma]$ supported on representatives of $\Oc$; let $\mchar_{\Oc}^k$ be the character of the complex permutation action on the monomials of positive degree $k$ supported on any representative of $\Oc$. We claim that 
\begin{equation} \label{eq:mchar_proper}
    \mchar_{\Oc}^k = \fchar_{\Oc}^{\Sigma} \cdot \binom{k-1}{\mathrm{rk}(\Oc)-1}.
\end{equation}

To prove \eqref{eq:mchar_proper}, by \Cref{rem:evaluation at every g} it is enough to prove that $\mchar_{\Oc}^k$ and $\fchar_{\Oc}^{\Sigma} \cdot \binom{k-1}{\mathrm{rk}(\Oc)-1}$ coincide when evaluated at any $g \in G$. Since both $\mchar_{\Oc}^k$ and $\fchar_{\Oc}^{\Sigma}$ are permutation characters, their evaluation at $g$ yields the number of fixed points of the associated permutation. In particular, $\mchar_{\Oc}^k(g)$ is the number of monomials $m$ of degree $k$ such that $G \cdot \mathrm{supp}(m) = \Oc$ and $gm = m$; by properness, this is also the number of monomials $m$ of degree $k$ supported on a face $\sigma \in \Sigma$ such that $G\sigma = \Oc$ and $g\sigma = \sigma$. The latter quantity equals $\fchar_{\Oc}^{\Sigma}(g)\cdot \binom{k-1}{\mathrm{rk}(\Oc)-1}$, as desired.

A routine computation then yields that

\[
\begin{split}
\eqhilb{\rho}{\mathbb{C}[\Sigma]}{t} &= 1 + \sum_{k \geq 1}\left(\sum_{\substack{\Oc \in \Sigma/G\\ 0 < \mathrm{rk}(\Oc)\leq k}}\mchar^k_{\Oc}\right) t^k\\ 
&= 1 + \sum_{\substack{\Oc \in \Sigma/G \\ \mathrm{rk}(\Oc)>0}}\sum_{k \geq \mathrm{rk}(\Oc)}\mchar^k_{\Oc} \cdot t^k\\
&= 1 + \sum_{\substack{\Oc \in \Sigma/G \\ \mathrm{rk}(\Oc)>0}}\sum_{k \geq \mathrm{rk}(\Oc)} \fchar_{\Oc}^{\Sigma} \cdot \binom{k-1}{\mathrm{rk}(\Oc)-1} \cdot t^k\\ 
&= 1 + \sum_{\substack{\Oc \in \Sigma/G \\ \mathrm{rk}(\Oc)>0}}\fchar_{\Oc}^{\Sigma} \cdot t^{\mathrm{rk(\Oc)}}\sum_{k \geq \mathrm{rk}(\Oc)}  \binom{k-1}{\mathrm{rk}(\Oc)-1} \cdot t^{k-\mathrm{rk}(\Oc)}\\ 
&= 1 + \sum_{\substack{\Oc \in \Sigma/G \\ \mathrm{rk}(\Oc)>0}}\fchar_{\Oc}^{\Sigma} \cdot t^{\mathrm{rk(\Oc)}} \sum_{j \geq 0} \binom{\mathrm{rk}(\Oc)+j-1}{j} \cdot t^j\\
&= \sum_{\Oc \in \Sigma/G}\fchar_{\Oc}^{\Sigma} \cdot \frac{t^{\mathrm{rk(\Oc)}}}{(1-t)^{\mathrm{rk}(\Oc)}}.
\end{split}
\]

\item This is \cite[Lemma 1.2]{StembridgeWeyl}: we reproduce it here with our notation. Let $g \in G$. Since $\rho$ is a proper action, the set $\Sigma^g$ is a simplicial complex by \Cref{lem:proper action on sc}. Then

\begin{equation} \label{eq:hilbert Sigma^g}
\mathrm{Hilb}(\mathbb{C}[\Sigma^g], t) = \sum_{i=0}^{\dim(\Sigma^g)+1}f^{\Sigma^g}_{i-1} \frac{t^i}{(1-t)^i} = \sum_{\Oc \in \Sigma/G}f^{\Sigma^g}_{\Oc} \frac{t^{\mathrm{rk}(\Oc)}}{(1-t)^{\mathrm{rk}(\Oc)}},
\end{equation}

\noindent where by $f^{\Sigma^g}_{\Oc}$ we mean the number of faces $\sigma \in \Sigma^g$ such that $G\sigma = \Oc$. Now note that $f^{\Sigma^g}_{\Oc}$ equals $\fchar^{\Sigma}_{\Oc}(g)$, since both quantities equal the number of faces $\sigma \in \Sigma$ such that $g\sigma = \sigma$ and $G\sigma = \Oc$. Comparing \eqref{eq:hilbert Sigma^g} with the evaluation at $g$ of \eqref{eq:eqhilb_proper} yields the claim.
\end{enumerate}
\end{proof}

\begin{remark} Note that \eqref{eq:hilbfaces_general} might fail when the action $\rho$ on $\Sigma$ by simplicial automorphisms is not proper. As an example, pick as $\Sigma$ a single $1$-simplex and let $G = \mathbb{Z}_2$ act on $\Sigma$ by swapping its vertices. Then, analyzing the monomials in $\mathbb{C}[\Sigma]$, one has that 
\[\eqhilb{\rho}{\mathbb{C}[\Sigma]}{t} = 1 + (1+\chi)t + (2+\chi)t^2 + \boldsymbol{(2+2\chi)}t^3 + \ldots\]

On the other hand, $\fchar^{\Sigma}_{-1} = \fchar^{\Sigma}_1 = 1$ and $\fchar^{\Sigma}_0 = 1+\chi$, whence 
\[\begin{split}\sum_{i=0}^{2}\fchar^{\Sigma}_{i-1}\frac{t^i}{(1-t)^i} &= 1 + (1+\chi)\cdot\frac{t}{1-t} + 1\cdot\frac{t^2}{(1-t)^2}\\ &= 1 + (1+\chi)t + (2+\chi)t^2 + \boldsymbol{(3+\chi)}t^3 + \ldots\end{split}\]
\end{remark}

We now recall a useful characterization of linear system of parameters for Stanley--Reisner rings, known as the \emph{Kind--Kleinschmidt criterion}.

\begin{notation}
Let $\mathbb{K}$ be a field, let $\Sigma$ be a finite abstract simplicial complex. If $\theta = \sum_{v 
\in V(\Sigma)}\lambda_v x_v \in \field[\Sigma]_1$ and $\sigma \in \Sigma \setminus \{\emptyset\}$, we write $\theta\vert_\sigma$ to denote $\sum_{v \in \sigma}\lambda_v x_v$, i.e.,~the linear form obtained from $\theta$ by picking only the terms containing variables indexed by elements of $\sigma$.
\end{notation}

\begin{lemma}[Kind--Kleinschmidt criterion, {\cite[Lemma III.2.4]{StanleyCombCommAlg}}] \label{lem:kindkleinschmidt}
Let $\mathbb{K}$ be a field and let $\Sigma$ be a finite abstract simplicial complex of dimension $d \geq 0$. The linear forms $\theta_0,\ldots,\theta_d \in \field[\Sigma]_1$ form a linear system of parameters for $\field[\Sigma]$ if and only if for every facet (or, equivalently, every nonempty face) $\sigma$ of $\Sigma$ the $\mathbb{K}$-vector subspace
$$
\langle \theta_0\vert_\sigma, 
\theta_1\vert_\sigma,
\ldots
\theta_d\vert_\sigma \rangle \subseteq \field[\Sigma]_1
$$
has dimension equal to $\vert \sigma\vert$.
\end{lemma}

\begin{remark}
The existence of a linear system of parameters for $\field[\Sigma]$ (and, more generally, for any standard graded $\mathbb{K}$-algebra) is guaranteed whenever $\mathbb{K}$ is an \emph{infinite} field. When instead $\mathbb{K}$ is finite, a linear system of parameters for $\field[\Sigma]$ might not exist. For instance, take $\mathbb{K} = \mathbb{F}_2$ and let $\Sigma$ be the complete graph on four elements, seen as a $1$-dimensional simplicial complex.
\end{remark}

We are now ready to prove the key technical result of this section: a new characterization of translativity for group actions on simplicial complexes.

\begin{theorem} \label{prop:translativity via lsop}
Let $\mathbb{K}$ be an infinite field and let $\Sigma$ be a finite nonempty abstract simplicial complex of dimension $d$ with an action of a group $G$ via simplicial automorphisms. The action is translative if and only if the ring $\field[\Sigma]$ has a linear system of parameters $\{\theta_i\}_i$ where each $\theta_i$ lies in the invariant subring $\field[\Sigma]^G$, i.e., $g\theta_i=\theta_i$ for all $i$ and $g$.
\end{theorem}

\begin{proof} If $\Sigma$ is the complex consisting only of the empty face, there is nothing to prove. Assume thus that $d \geq 0$. By \Cref{lem:translative action on sc}, the action of $G$ is translative if and only if it preserves a proper coloring of $\Sigma$.

Suppose that the action of $G$ preserves a proper $(c+1)$-coloring $\gamma\colon V(\Sigma) \to \{0,\ldots,c\}$, where $c \geq d$. For all $i=0,\ldots,d$ define
$$
\theta_i:= \sum_{v\in V(\Sigma)} \gamma(v)^i x_v,
$$
where we are setting $0^0 = 1$ by convention. Since the coloring $\gamma$ is preserved by the group action and the action permutes the vertices of $\Sigma$, we have that
$$
g\theta_i = \sum_{v\in V(\Sigma)} \gamma(v)^i x_{gv} = \sum_{v\in V(\Sigma)} \gamma(gv)^i x_{gv} = \theta_i
$$
for all $g\in G$ and all $i$. It remains to be shown that  $\theta_0,\ldots,\theta_d$ is a linear system of parameters. For this, consider any nonempty face $\sigma=\{v_0,\ldots,v_k\}$ of $\Sigma$. Then the subspace 
$\langle \theta_0\vert_\sigma, 
\theta_1\vert_\sigma,
\ldots
\theta_d\vert_\sigma \rangle$
is the image of the $(k+1) \times (d+1)$ matrix
$$
T:=
\begin{pmatrix}
1 & \gamma(v_0) & \gamma(v_0)^2 & \ldots & \gamma(v_0)^d \\
1 & \gamma(v_1) & \gamma(v_1)^2 & \ldots & \gamma(v_1)^d \\
\vdots & \vdots & \vdots & & \vdots \\
1 & \gamma(v_k) & \gamma(v_k)^2 & \ldots & \gamma(v_k)^d \\
\end{pmatrix}
$$
Since $k\leq d$, the rank of the matrix is at least equal to the rank of its leftmost maximal submatrix, 
which is a Vandermonde matrix of size $(k+1)\times (k+1)$  with determinant
$$
\prod_{0\leq i < j \leq k} (\gamma(v_i) - \gamma(v_j))\neq 0
$$
(here we use that $\gamma$ is a proper coloring, implying $\gamma(v_i)\neq \gamma(v_j)$ for all $v_i\neq v_j$). Thus the matrix $T$ has maximal rank, implying that $\dim_{\mathbb{K}}\langle \theta_0\vert_\sigma, 
\theta_1\vert_\sigma,
\ldots
\theta_d\vert_\sigma \rangle = k+1 = \vert \sigma\vert$.

For the reverse implication, assume that we are given a linear system of parameters $\theta_0,\ldots,\theta_d$ such that $g\theta_i=\theta_i$ for all $i$ and $g$. This implies that in every $\theta_i$ the variables corresponding to monomials in the same $G$-orbit must have the same coefficient. 
Thus, if $w_0,\ldots,w_c$ is a system of representatives for the orbits of the action of $G$ on $V(\Sigma)$, then every $\theta_i$ can be written as
$$
\theta_i = \sum_{j=0}^c a^{(i)}_j\left(\sum_{g\in G} x_{gw_j}\right)
$$
for some $a^{(i)}_0,\ldots a^{(i)}_c\in \mathbb{K}$. Now define a coloring of $\Sigma$ by assigning to every vertex $v$ the unique color $\gamma(v)\in \{0,\ldots, c\}$ such that $v\in Gw_{\gamma(v)}$. This coloring is obviously preserved by $G$. In order to prove that the action of $G$ is translative, it is enough to prove that the coloring $\gamma$ is proper. If $\Sigma$ is $0$-dimensional, this is trivially true. Assume thus that $\Sigma$ has dimension at least $1$ and consider any nonempty edge $\sigma=\{v,w\} \in \Sigma$. By the Kind--Kleinschmidt criterion (\Cref{lem:kindkleinschmidt}), the vector subspace $\langle \theta_0\vert_\sigma, 
\theta_1\vert_\sigma,
\ldots
\theta_d\vert_\sigma \rangle$ has dimension $2$. By construction, the coefficient of $x_v$ (respectively, $x_w$) in $\theta_i$ is $a^{(i)}_{\gamma(v)}$ (respectively, $a^{(i)}_{\gamma(w)}$). As a consequence, the $2$-dimensional subspace 
$\langle \theta_0\vert_\sigma, 
\theta_1\vert_\sigma,
\ldots
\theta_d\vert_\sigma \rangle$ is the image of the $2 \times (d+1)$ matrix

\[\begin{pmatrix}
    a^{(0)}_{\gamma(v)} & a^{(1)}_{\gamma(v)} & \cdots & a^{(d)}_{\gamma(v)}\\[0.6em]
    a^{(0)}_{\gamma(w)} & a^{(1)}_{\gamma(w)} & \cdots & a^{(d)}_{\gamma(w)}
\end{pmatrix}\!,\]

\noindent and this matrix must have rank $2$. This implies immediately that $\gamma(v) \neq \gamma(w)$.
\end{proof}

\begin{remark}
The proof of \Cref{prop:translativity via lsop} shows that the statement holds also for those finite fields $\mathbb{K}$ containing at least as many distinct elements as the number of vertices of $\Sigma$. 
\end{remark}

We are now in the position to apply \Cref{cor:effective numerator}.

\begin{theorem} \label{thm:effective numerator for SR}
    Let $\Sigma$ be a finite abstract simplicial complex that is $d$-dimensional and Cohen--Macaulay over $\mathbb{C}$, let $G$ be a finite group, and let $\rho$ be a translative action of $G$ on $\Sigma$. Let $\theta_0, \ldots, \theta_d$ be a linear system of parameters for $\mathbb{C}[\Sigma]$ such that each $\theta_i$ lies in the invariant ring $\mathbb{C}[\Sigma]^G$. Then 
    \begin{equation} \label{eq:translative rational eqHilb}
    \eqhilb{\rho}{\mathbb{C}[\Sigma]}{t} = \frac{\eqhilb{\overline{\rho}}{\mathbb{C}[\Sigma]/J}{t}}{(1-t)^{d+1}},
    \end{equation}
    where $J$ is the $G$-invariant ideal of $\mathbb{C}[\Sigma]$ generated by $\theta_0, \ldots, \theta_d$ and $\overline{\rho}$ is the $G$-action induced by $\rho$ on the quotient $\mathbb{C}$-algebra $\mathbb{C}[\Sigma]/J$. In particular, the numerator of the rational expression in \eqref{eq:translative rational eqHilb} is effective.
\end{theorem}

\begin{remark} \label{rem:each h_i is effective}
Combining \Cref{thm:effective numerator for SR} with equation \eqref{eq:hilbfaces_general} in \Cref{prop:eqhilb_proper} yields that, under the hypotheses of \Cref{thm:effective numerator for SR}, each $\hchar_i^{\Sigma}$ coincides with the complex character of the action of $\rho$ on the degree $i$ part of $\mathbb{C}[\Sigma]/J$; in particular, each $\hchar_i^{\Sigma}$ is effective.
\end{remark}

For a simple example where \Cref{thm:effective numerator for SR} can be used, we refer the reader to \Cref{ex:equivariant Artinian reduction}. Indeed, the ring $R$ in \Cref{ex:equivariant Artinian reduction} is the Stanley--Reisner ring over $\mathbb{C}$ of the $0$-dimensional (and hence Cohen--Macaulay) simplicial complex $\Sigma$ having two points as its facets.

\subsection{The colorful approach} \label{subsec:eqHilb_colorful}

In this section, closely following the paper \cite{AdamsReiner} by Adams and Reiner, we wish to extract information from a suitable \emph{multigraded} structure of $\mathbb{C}[\Sigma]$, when a finite group $G$ acts translatively on $\Sigma$. To do this, we will quickly recap the approach by Adams and Reiner for the reader's convenience. For simplicity, like in most of \Cref{subsec:eqHilb_SC_CM}, we will be dealing only with Stanley--Reisner rings over $\mathbb{C}$, and will only consider the equivariant Hilbert series of $\mathbb{C}[\Sigma]$ with respect to the standard $\mathbb{N}$-grading, although multigraded versions are available (see for instance the proof of \Cref{prop:hilbfaces}). Since we deal with complex representations of finite groups, we will use characters to encode the group action information, as we explained in \Cref{subsec:eqHilb in general}. We refer the reader interested in a more general setting to the original paper \cite{AdamsReiner}.

\begin{remark}\label{rem:chains}
Recall that the construction of a complex of chains for a simplicial complex $\Sigma$ depends on the choice of an orientation of the simplices of $\Sigma$, which can be obtained by choosing a total order $\prec$ on $V(\Sigma)$. Then the vector space 
$
C_i(\Sigma,\mathbb C)
$
of $i$-chains of $\sigma$ is the vector space with standard basis $\{e_\sigma \mid \sigma\in \Sigma, \vert \sigma\vert = i+1\}$. These turn into a chain complex with boundary maps defined on basis vectors as follows:
$$
\partial_i: 
C_i(\Sigma,\mathbb C)\to C_{i-1}(\Sigma,\mathbb C),\quad
e_{\{v_0\prec v_1 \prec\ldots\prec v_i\}} \mapsto \sum_{j=0}^i 
(-1)^j e_{\{v_0\prec \ldots \prec v_{j-1} \prec v_{j+1}\prec\ldots\prec v_i\}}.
$$
Any action of a group $G$ on a simplicial complex $\Sigma$ induces linear representations of $G$ on each $C_i(\Sigma,\mathbb C)$ that are compatible with the boundary maps (thus inducing a linear representation on the reduced homology of $\Sigma$). We now show that when the action is translative, one can choose an orientation of the simplices of $\Sigma$ such that the representation $\rho_i$ on each $C_i(\Sigma,\mathbb C)$ is by permutation of coordinates.

Indeed, let $\mathcal{V}_1,\ldots,\mathcal{V}_k$ denote the $G$-orbits of vertices of $\Sigma$ under the action of $G$. Translativity implies that every $\sigma\in \Sigma$ contains at most one vertex from each $\mathcal{V}_i$. Now we can choose a total ordering $\prec$ on $V(\Sigma)$ such that $v\prec w$ whenever $v\in \mathcal{V}_i$ and $w\in \mathcal{V}_j$ for $i<j$.  Now construct the chain complex with respect to this ordering. If $\sigma=\{v_0\prec v_1 \prec\ldots\prec v_i\}\in \Sigma$, translativity implies that all $v_i$ are from different orbits. Our choice of $\prec$ implies that for every $g\in G$ we have
$g\sigma=\{gv_0\prec gv_1 \prec\ldots\prec gv_i\}$, as desired.
\end{remark}

Let $\Sigma$ be a finite nonempty abstract simplicial complex and let $\rho$ be an action of a finite group $G$ on $\Sigma$ by simplicial automorphisms. By \Cref{lem:translative action on sc}, we know that $\rho$ is translative if and only if there exists a proper coloring $\gamma$ of $\Sigma$ that is preserved by $\rho$. For the rest of the section we will hence assume that $\Sigma$ comes with a proper coloring $\gamma\colon V(\Sigma) \to \Gamma$, and $\rho$ is a color-preserving action with respect to $\gamma$. Given any $S \subseteq \Gamma$, we will denote by $\Sigma\vert_S$ the subcomplex of $\Sigma$ consisting of those faces $\sigma \in \Sigma$ with $\gamma(\sigma) \subseteq S$. Note that the action of $G$ on $\Sigma$ restricts to a translative action of $G$ on each $\Sigma\vert_S$.

\begin{remark}[Equivariant topology of $\Sigma\vert_S$]

Since $\gamma$ is a proper coloring, the colors of the vertices of any given face $\sigma \in \Sigma$ form a set $S \subseteq \Gamma$ containing precisely $|\sigma|$ colors. Since $\rho$ preserves $\gamma$, in particular it permutes the set $\{\sigma\in \Sigma\mid \gamma(\sigma)=S\}$ of $S$-colored faces.

Let us write $\fchar^\Sigma_S$ for the character of the associated complex permutation representation. Note that $\fchar^\Sigma_S$ is also the character of the permutation representation of $G$ on the standard basis of the vector space $C_{\vert S \vert -1}(\Sigma\vert_S,\mathbb C)$ of $(\vert S \vert - 1)$-chains of $\Sigma\vert_S$ constructed with the orientation chosen as in \Cref{rem:chains}. 

 Let $\hchar^{\Sigma}_S$ denote the virtual character defined by 
 \begin{equation}\label{defhcharS}
 \hchar^{\Sigma}_S \coloneqq \sum_{T \subseteq S}(-1)^{|S|-|T|}\fchar^{\Sigma}_T.
 \end{equation}
 Since the action of $G$ is compatible with the boundary maps, 
$G$ also acts on each reduced homology group $\widetilde{H}_i(\Sigma\vert_S,\mathbb C)$ defining a representation whose character we denote by $\homchar_{i}^{\Sigma\vert _S}$. We define
$$\homchar_S^\Sigma:=\sum_{i\geq -1} (-1)^i \homchar^{\Sigma\vert_S}_i.
$$
 \end{remark}

 The following lemma is folklore. We include a sketch of proof following the argument given in \cite[Theorem 1.1]{StanleyActions} for the case of order complexes of ranked posets.

 \begin{lemma} \label{lem:baclawskibjorner} Let $\Sigma$ be a finite abstract simplicial complex with a proper coloring $\gamma\colon V(\Sigma)\to \Gamma$ and let $\rho$ be a color-preserving action of the finite group $G$ on $\Sigma$. Then for every $S\subseteq \Gamma$ we have
 $$
 \hchar_S^\Sigma = \homchar_S^\Sigma.
 $$
 In particular, if the reduced homology of $\Sigma\vert_S$ is concentrated in a single degree $k$ (e.g., when $\Sigma_S$ is Cohen-Macaulay), then $\hchar^\Sigma_S = \homchar^{\Sigma\vert_S}_k$ is the character of a genuine representation. 
 \end{lemma}
 \begin{proof}
 The evaluation at $g\in G$ of the character $\homchar^\Sigma:=\sum_{i\geq -1} (-1)^i \homchar^\Sigma_i$ is what is referred to as the Lefschetz number of the simplicial map induced by $g$ in \cite{BaBj}. In  \cite[Theorem 1.1]{BaBj} it is proved that if the action of $G$ fixes pointwise every setwise-fixed simplex. Then $\hchar_S^\Sigma(g)$ (as defined in \eqref{defhcharS}) is the Euler characteristic of the fixed complex $(\Sigma\vert_S)^g$, and this equals the Lefschetz number $\homchar^{\Sigma}_S(g)$. The assumption about fixing simplices is satisfied for us since color-preserving actions are translative by \Cref{lem:translative action on sc}, and hence also proper by \Cref{lem:translative implies proper}. Moreover, as was pointed out in \cite[Theorem 1.1]{StanleyActions}, the proof of \cite[Theorem 1.1]{BaBj} works for reduced homology as well. We conclude with \Cref{rem:evaluation at every g}.
 \end{proof}

\begin{remark}[Equivariant flag $h$- and $f$-vectors]
The collections $(\fchar^\Sigma_S)_{S \subseteq \Gamma}$ and $(\hchar^\Sigma_S)_{S \subseteq \Gamma}$ are equivariant versions of the \emph{flag $f$-vector} and \emph{flag $h$-vector} of the pair $(\Sigma, \gamma)$, and give those back when evaluated at the identity element of $G$. For a reference to non-equivariant flag $f$- and $h$-vectors in the balanced case see, e.g.,~\cite[Section III.4]{StanleyCombCommAlg}.
\end{remark}

Just like in the non-equivariant case, the flag $f$-vector of $\Sigma$ refines the $f$-vector of $\Sigma$; indeed, one has that $\fchar_{i-1}^{\Sigma} = \sum_{|S|=i}\fchar_S^{\Sigma}$ for every $i \in \{0, 1, \ldots, \dim(\Sigma)+1\}$.

\begin{proposition} \label{prop:hilbfaces}
    Let $\Sigma$ be a finite abstract simplicial complex with a proper coloring $\gamma\colon V(\Sigma)\to \Gamma$ and let $\rho$ be a color-preserving action of the finite group $G$ on $\Sigma$.

    \begin{enumerate}
    \item The following equality holds in $R_G[[t]]$.
    \begin{equation} \label{eq:eqhilb_translative}
    \eqhilb{\rho}{\mathbb{C}[\Sigma]}{t} = 
    \sum_{S \subseteq \Gamma} \fchar^{\Sigma}_{S} \frac{t^{|S|}}{(1-t)^{|S|}} = \frac{\sum_{S \subseteq \Gamma} \hchar^\Sigma_S \cdot t^{|S|}}{(1-t)^{|\Gamma|}}.
\end{equation}

    \item If $|\Gamma|=\dim(\Sigma)+1$, then the equivariant flag $h$-vector of $\Sigma$ refines the equivariant $h$-vector of $\Sigma$, i.e.,~$\hchar_i^{\Sigma} = \sum_{\mathrm{rk}(S)=i}\hchar_S^{\Sigma}$ for every $i \in \{0, 1, \ldots, \vert \Gamma\vert\}$.
    \end{enumerate}

\end{proposition}
\begin{proof}
Part (i) comes from specializing \cite[Proposition 2.5]{AdamsReiner} from the $\mathbb{N}^{|\Gamma|}$-graded to the $\mathbb{N}$-graded setting. Note that the first equality in \eqref{eq:eqhilb_translative} can also be derived from equation \eqref{eq:eqhilb_proper} in \Cref{prop:eqhilb_proper} by noting that $f_S^{\Sigma} = \sum_{\gamma(\Oc) = S}f_{\Oc}^{\Sigma}$ for every $S \subseteq \Gamma$. Part (ii) is a direct consequence of part (i) and \Cref{prop:eqhilb_proper}. 
\end{proof}

Adams and Reiner build on the last statement of \Cref{lem:baclawskibjorner} in order to study the case when $\Sigma$ is the order complex of a Cohen--Macaulay poset. We note below that the same reasoning can be applied verbatim to the more general case when $\Sigma$ is a Cohen--Macaulay balanced complex.

\begin{lemma}
\label{lem:balanced pseudoeffectiveness}
Let $\Sigma$ be a finite abstract simplicial complex that is both balanced and Cohen--Macaulay over $\mathbb{C}$. Let $\gamma\colon V(\Sigma) \to \Gamma$ be a proper coloring of $\Sigma$ with $|\Gamma| = \dim(\Sigma)+1$ and let $\rho$ be a color-preserving action of the finite group $G$ on $\Sigma$. Then $\hchar_S^{\Sigma}$ is effective for every choice of $S \subseteq \Gamma$.
\end{lemma}
\begin{proof}
Let $S \subseteq \Gamma$. By \cite[Theorem 4.3]{StanleyBalanced}, the $(|S|-1)$-dimensional color-selected subcomplex $\Sigma|_S$ is Cohen--Macaulay, and hence $\widetilde{H}_{i}(\Sigma|_S, \mathbb{C})$ vanishes for every $i < |S|-1$. Then $\hchar_S^{\Sigma}$ is the character of the $G$-representation on $\widetilde{H}_{|S|-1}(\Sigma|_S, \mathbb{C})$ from \Cref{lem:baclawskibjorner}.
\end{proof}

\begin{remark} \label{rem:negative h_S}
Note that the claim of \Cref{lem:balanced pseudoeffectiveness} does not necessarily hold if the coloring preserved by the translative action $\rho$ is not balanced. For instance, consider a hexagon endowed with the antipodal (translative) $\mathbb{Z}_2$-action. The unique associated proper coloring $\gamma$ assigns a different color to each of the three pairs of antipodal vertices. Then, for every set $S$ of two distinct colors, one has that $\hchar_S = 1 - 2(1+\chi_{\sgn}) + (1+\chi_{\sgn}) = -\chi_{\sgn}$. See also \Cref{rem:equivariant Hilbert hexagon} below.
\end{remark}

Putting together \Cref{prop:hilbfaces} and \Cref{lem:balanced pseudoeffectiveness} yields the following result, which can be thought of as a refinement of \Cref{cor:effective numerator}:

\begin{corollary}
    Let $\Sigma$ be a finite abstract $d$-dimensional simplicial complex that is both balanced and Cohen--Macaulay over $\mathbb{C}$. Let $\gamma\colon V(\Sigma) \to \Gamma$ be a proper coloring of $\Sigma$ with $|\Gamma| = d+1$ and let $\rho$ be a color-preserving action of the finite group $G$ on $\Sigma$. Then
    \begin{equation} \label{eq:rational eqHilb for balanced Sigma}
        \eqhilb{\rho}{\mathbb{C}[\Sigma]}{t} = \frac{\sum_{S \subseteq \Gamma} \hchar^\Sigma_S \cdot t^{|S|}}{(1-t)^{d+1}},
    \end{equation}
    and each $\hchar^\Sigma_S$ is effective. In particular, the numerator of the rational expression in \eqref{eq:rational eqHilb for balanced Sigma} is effective.
\end{corollary}

\begin{remark} \label{rem:equivariant Hilbert hexagon}
We remark here that, when the action $\rho$ preserves a proper coloring $\gamma\colon V(\Sigma) \to \Gamma$ with $|\Gamma| > \dim(\Sigma)+1$, it is still true by \Cref{thm:effective numerator for SR} that the rational expression of $\eqhilb{\rho}{\mathbb{C}[\Sigma]}{t}$ with denominator $(1-t)^{\dim(\Sigma)+1}$ has an effective numerator. This does not contradict \Cref{rem:negative h_S}; as an example, in the case of the $\mathbb{Z}_2$-equivariant Hilbert series of the $1$-dimensional hexagon $\Sigma$ from \Cref{rem:negative h_S}, we have that \[\eqhilb{\rho}{\mathbb{C}[\Sigma]}{t} = \frac{1+3\chi_{\sgn}t-3\chi_{\sgn}t^2-t^3}{(1-t)^3} = \frac{1+(1+3\chi_{\sgn})t+t^2}{(1-t)^2}.\]
\end{remark}

\section{An application to equivariant Ehrhart series}\label{sec:main}

\subsection{The equivariant Ehrhart series of a lattice polytope} \label{sub:ee}

\def\M{\widetilde{M}}
\def\Mo{M}

We recall the setup of \cite[Section 2.2]{StapledonNew}, to which we refer for proofs and background. 

\newcommand{\aff}[1]{\operatorname{Aff}(#1)}

Let $\Mo\simeq \mathbb Z^d$ be a lattice (i.e., a finitely generated free abelian group). An affine transformation of $\Mo$ is by definition a function $\phi$ of the form $x\mapsto Ax + b$ with $A\in \mathrm{GL}(\Mo)$ and $b \in \Mo$. The affine transformations form a group $\aff{\Mo}$.

Now let $\M:=\Mo \oplus \mathbb Z \simeq \mathbb{Z}^{d+1}$ and, for every $\phi\in \aff{\Mo}$ and every $x\oplus u \in \M$, let $\widetilde{\phi} (x\oplus u):=\phi(x) \oplus u$. Then $\widetilde{\phi}\in \mathrm{GL}(\M)$, and so every affine representation $\rho$ defines a linear representation $\widetilde{\rho}$. Tensoring with $\mathbb C$, any {\em affine representation} $\rho: G\to \operatorname{Aff}(\mathbb C^d)$ determines a (linear) representation $\widetilde{\rho}:G\to\gl{d+1}{\mathbb C}$ of $G$ on $\M \otimes \mathbb C \simeq \mathbb C^{d+1}$ whose image is in the subgroup of transformations preserving the affine subspace $\mathbb C^d \times\{1\} \subseteq \M \otimes \mathbb C$, see \cite[Section 2.3]{StapledonNew}.

\bigskip
 
 Let $P$ be an $\Mo$-lattice polytope that is invariant with respect to the $G$-action by affine transformations given by $\rho$. We will identify $P$ with the  polytope $P\oplus \{1\}\subseteq (\Mo\otimes \mathbb R) \oplus \{1\}$, which is $G$-invariant with respect to $\widetilde{\rho}$. Then, for every $m\in \mathbb N$, the polytope $mP\subseteq (\Mo\otimes \mathbb R) \oplus \{m\}$ is also $G$-invariant (with respect to $\widetilde{\rho}$), see \cite[Section 2.4]{StapledonNew}. In particular, for every $m$ the group $G$ acts as a permutation on the set $mP\cap \M$. Let $\chi_{mP}\in \basering{G}$ be the character of the associated complex permutation representation. 

\begin{definition}[{\cite[Definition 2.10]{StapledonNew}}] \label{def:eqehr} 
Let $\Mo\subseteq \mathbb R^{d}$ be a lattice and let $\rho\colon G\to \aff{\Mo}$ be an affine representation. With the notation above, the \emph{equivariant Ehrhart series} of a $G$-invariant $\Mo$-lattice polytope $P\subseteq \mathbb R^d$ is
$$
\eqehr{\rho}{P}{t}:=1 +\sum_{m\geq 1} \chi_{mP} t^m \in \basering{G}[[t]].
$$
\end{definition}

It is natural to look for a rational expression for the equivariant Ehrhart series of $P$, but it is not a priori obvious which denominator to choose. Backed by geometric and combinatorial considerations, Stapledon proposes to use the element $\det(\mathrm{Id} - \widetilde{\rho} t)$, where

\[\det(\mathrm{Id} - \widetilde{\rho} t) := \sum_{i=0}^{d+1}(-1)^i\chi_{\wedge^i(\widetilde{M} \otimes \mathbb{C})}t^i.\]

One then writes

\begin{equation}
    \eqehr{\rho}{P}{t} = \frac{h^*(P,\rho;t)}{\det(\mathrm{Id} - \widetilde{\rho} t)},
\end{equation}

\noindent where $h^*(P,\rho;t)$ is a \emph{power series} in $R_G[[t]]$ and not necessarily a polynomial! Note that, when all coefficients of $h^*(P,\rho;t)$ are effective, the power series $h^*(P,\rho;t)$ is actually a polynomial, as follows from considering its evaluation at $\mathrm{id}_G$: see e.g.~\cite[Lemma 2.8]{EKS}. It is currently open whether polynomiality and effectiveness of $h^*(P,\rho;t)$ are equivalent conditions; this is conjectured to be true in \cite[Conjecture 1.2]{StapledonNew}. 

\subsection{Main theorem}

We can now state and prove the main contribution of the present paper to equivariant Ehrhart theory.

\begin{theorem}\label{thm:EBM}
Let $M \subseteq \mathbb{R}^d$ be a lattice and let $P \subseteq \mathbb{R}^d$ be a $d$-dimensional $M$-lattice polytope with a unimodular lattice triangulation $\Delta$. Let $\rho\colon G \to \mathrm{Aff}(M)$ be an action of a finite group $G$ by affine transformations of $M$ preserving $\Delta$. Then the following equality holds in $\basering{G}[[t]]$:
\begin{equation} \label{eq:ehrhart_hilbert}
\eqehr{\rho}{P}{t} = \eqhilb{\rho}{\mathbb{C}[\usc{\Delta}]}{t}.
\end{equation}
Assume further that the action $\rho$ is translative on $\Delta$ and let $\gamma\colon V(\Delta) \to \Gamma$ be a proper coloring of $\Delta$ preserved by $\rho$. Then 
\begin{equation} \label{eq:translative_ehrhart}
    \eqehr{\rho}{P}{t} = \frac{\sum_{S \subseteq \Gamma} \hchar^{\Delta}_{S} t^{|S|}}{(1-t)^{|\Gamma|}} = \frac{\sum_{i=0}^{d+1}\hchar_i^{\Delta}t^i}{(1-t)^{d+1}},
\end{equation}
where each $\hchar^{\Delta}_i$ (but not necessarily each $\hchar^{\Delta}_S$) is effective. If moreover $|\Gamma|=d+1$, i.e.,~if $\gamma$ is balanced, then each $\hchar^{\Delta}_S$ in \eqref{eq:translative_ehrhart} is effective.

\end{theorem}

\begin{remark}In \cite[Theorem 1.4]{StapledonNew} Stapledon proves that effectiveness  of $h^*(P,\rho;t)$ is implied by the existence of a $G$-invariant triangulation $\Delta$ of $P$, a condition that is always satisfied in the hypotheses of \Cref{thm:EBM}. If $\Delta$ is unimodular and $\rho$ is translative, we infer from \eqref{eq:translative_ehrhart} that \begin{equation}
h^*(P,\rho,t) = \frac{\det(\mathrm{Id}-\widetilde{\rho} t)\sum_{i=0}^{d+1}\hchar_i^{\Delta}t^i}{(1-t)^{d+1}}.
\end{equation}
Due to Stapledon's result, $h^*(P,\rho,t)$ is effective and polynomial; in particular, the polynomial $\det(\mathrm{Id}-\widetilde{\rho} t) \cdot \sum_{i=0}^{d+1}\hchar_i^{\Delta}t^i \in \basering{G}[t]$ is divisible by $(1-t)^{d+1}$.
\end{remark}

\begin{remark} \label{rem:disclaimer}
In an earlier draft, equation \eqref{eq:ehrhart_hilbert} was proved only in the case when the action $\rho$ is translative. Victor Reiner informed us of unpublished notes by himself, Katharina Jochemko and Lukas Katth\"an, where the general version was stated and proved. We are grateful for the kind permission to include their argument.

 In the final stages of preparation of the manuscript we also became aware that a more general form of \eqref{eq:ehrhart_hilbert}, without unimodularity assumptions, can be found in \cite[Proposition 4.40]{StapledonNew}. See also \cite[Remark 4.41]{StapledonNew}. 
\end{remark}

\begin{proof}
As usual (see \Cref{sub:ee}) we think of $P$, and thus of $\Delta$, as lying in $M\oplus\{1\}\subseteq \mathbb R^d\oplus \mathbb R$. The affine action $\rho$ on $M$ extends to a linear action $\widetilde{\rho}$ on $\widetilde{M}=M\oplus \mathbb Z$ that fixes the last coordinate. Now, for every $m\geq 1$ the simplicial complex $m\Delta$ is a lattice triangulation of $mP$ preserved by $G$, thus every lattice point $p\in \widetilde{M}\cap mP$ is contained in the relative interior of the $m$-th dilation of exactly one simplex $\sigma_p\in \Delta$, say $\sigma_p=\conv\{v_1^p,\ldots,v_k^p\}$. This means that there are unique, strictly positive real numbers $\lambda^p_1,\ldots,\lambda^p_k$ summing up to $m$ and such that
$$
p=\lambda_1^pv_1^p+\ldots + \lambda_k^pv_k^p.
$$
Unimodularity of the triangulation implies that $\lambda_i^p\in \mathbb Z_{>0}$ for all $i$. This gives a well-defined bijection
$$
f_m: p\mapsto x_{v^p_1}^{\lambda^p_{1}}\cdots x_{v^p_k}^{\lambda^p_{k}}
$$
 between lattice points in $mP$ and monomials supported on $\usc{\Delta}$ of total degree $m$, i.e., a $\mathbb C$-basis of $\mathbb C [\usc{\Delta}]_m$. Now since $\Delta$ and $m\Delta$ are preserved by the group action, for every $g\in G$ we have that $\sigma_{g(p)}=g(\sigma_p)$. Uniqueness  of the $\lambda_i^p$ implies that $g(p)=\lambda_1^pg(v^p_1)+\ldots+ \lambda_k^pg(v^p_k)$, and thus
 $$g(f_m(p))=
 x_{g(v^p_1)}^{\lambda^p_{1}}\cdots x_{g(v^p_k)}^{\lambda^p_{k}}
  = f_m(g(p)).$$
This shows that $f_m$ induces an isomorphism of representations between the (complex) permutation representation on $\widetilde{M}\cap mP$ and the representation on $\mathbb C[\usc{\Delta}]_m$. This proves the (coefficientwise) equality $\eqehr{\rho}{P}{t} = \eqhilb{\rho}{\mathbb{C}[\usc{\Delta}]}{t}$.

Let us now assume that the action $\rho$ is translative on $\Delta$. Then, by \Cref{lem:translative action on sc}, we know that there exists a proper coloring $\gamma\colon V(\Delta) \to \Gamma$ that is preserved by $\rho$. The first equality in \eqref{eq:translative_ehrhart} follows from \Cref{prop:hilbfaces}. Since every translative action is proper by \Cref{lem:translative implies proper}, the second equality in \eqref{eq:translative_ehrhart} is a consequence of \Cref{prop:eqhilb_proper}. Since the geometric simplicial complex $\Delta$ is homeomorphic to a ball in $\mathbb{R}^d$, it is a Cohen--Macaulay complex over $\mathbb{R}$ and hence over $\mathbb{C}$ as well. The statement about the effectiveness of each $\hchar_i^{\Delta}$ is then a consequence of \Cref{thm:effective numerator for SR} and \Cref{rem:each h_i is effective}. Assuming further than the proper coloring $\gamma$ is balanced, one has that each $\hchar_S^{\Delta}$ in \eqref{eq:translative_ehrhart} is effective due to \Cref{lem:balanced pseudoeffectiveness}.

\end{proof}

\begin{remark} \label{rem:alternative proof of eBM}
We offer an alternative proof that $\eqehr{\rho}{P}{t} = \eqhilb{\rho}{\mathbb{C}[\usc{\Delta}]}{t}$ in the case when the action $\rho$ is proper on $\Delta$. By \Cref{rem:evaluation at every g}, it is enough to prove that 
\begin{equation} \label{eq:main equation evaluated at g}
\eqehr{\rho}{P}{t}(g) = \eqhilb{\rho}{\mathbb{C}[\usc{\Delta}]}{t}(g)
\end{equation}
for every choice of $g \in G$. 
But for every fixed $g\in G$ the simplicial complex $\Delta^g$ triangulates the lattice polytope $P^g$ by \Cref{lem:proper triangulation}, and then we have that
 $$
 \eqehr{\rho}{P}{t}(g)
 =
 \ehr{P^g}{t}
 =
 \hilb{\mathbb{C}[\usc{\Delta^g}]}{t}
 =
 \eqhilb{\rho}{\mathbb{C}[\usc{\Delta}]}{t}(g),
 $$
 where the first equality holds by \cite[Lemma 5.2]{Stapledon}, the second is the Betke--McMullen theorem (see, e.g., \cite[Theorem 10.3]{BeckRobins}), and the third is part (ii) of \Cref{prop:eqhilb_proper}.
\end{remark}

\section{Order polytopes}\label{sec:OP}

\newcommand{\triang}[1]{\mathscr T_{#1}}

In this section, when $X$ is a finite partially ordered set, we write $J(X)$ for the poset of lower order ideals of $X$ ordered by inclusion. Such a poset is ranked by the cardinality of its elements. 

\begin{definition}
Let $X$ be a finite partially ordered set. The \emph{order polytope} of $X$ is
$$
\orp(X):=\{f\in \mathbb{R}^X \mid 0\leq f(x) \leq f(x') \leq 1 \,\, \forall x\leq x'\in X\}.
$$
\end{definition}

Order polytopes were first defined by Stanley in \cite{StanleyTwo} and have been in the focus of active research ever since. Every order polytope has a natural regular unimodular triangulation $\triang{X}$ whose simplices are indexed by chains  $\omega\colon J_1 \subsetneq \ldots \subsetneq J_k$ of order ideals of $X$. Given such a chain $\omega$ and an element $x\in X$, write \[\omega(x)=\begin{cases}\min \{i \mid x\in J_i\} & \text{if }x \in J_k\\ k+1 & \text{if }x \in X \setminus J_k\end{cases}.\]  Then the simplex associated with $\omega$ is defined by 
$$
\Delta_{\omega} = \{f\in \mathbb R^X \mid \omega(x) \leq \omega(y)
\Rightarrow 0\leq f(x)\leq f(y) \leq 1\}
$$
This triangulation was first described in \cite[\S 5]{StanleyTwo}, where it is proved that the assignment $\omega \mapsto \Delta_\omega$ defines an isomorphism of abstract simplicial complexes 
$$\phi\colon\Delta(J(X)) \to \usc{\triang{X}}.$$

\begin{definition}
The {\em standard proper coloring} of $\triang{X}$ is defined by assigning to every vertex of $\triang{X}$ a number in $\{0,\ldots,\vert X \vert\}$ equal to the rank of the corresponding element of $J(X)$. Note that the standard proper coloring of $\triang{X}$ is balanced.
\end{definition}

 Recall that an automorphism of a poset $X$ is an order-preserving bijection $X\to X$ whose inverse is also order-preserving.

 \begin{remark} \label{rem:induced action}
 Let $G$ be a group acting by automorphisms on the poset $X$. Then
 \begin{enumerate}
 \item $G$ acts on $\mathbb R^X$ as follows: the automorphism $g: X\to X$ acts by sending $f\in \mathbb R^X$ to $gf:= f\circ g^{-1}\in \mathbb R^X$, given in coordinates by $gf(x):=f(g^{-1}(x))$ for all $x\in X$;
 \item $G$ acts on $\Delta(J(X))$ sending $\omega: J_1\subsetneq\ldots \subsetneq J_k$ to $g\omega: gJ_1\subsetneq\ldots\subsetneq gJ_k$.
 \end{enumerate}
 \end{remark}

\begin{lemma} \label{lem:translative_order_polytope} Let $G$ be any finite group of automorphisms of a finite poset $X$. Then $G$ acts on $\mathbb{R}^X$ by permuting coordinates as in \Cref{rem:induced action}.(i) and on $\Delta(J(X))$ as in \Cref{rem:induced action}.(ii).  The isomorphism $\phi$ is equivariant with respect to those actions. This induces a translative action on the regular unimodular triangulation $\triang{X}$ of $\orp(X)$ that preserves the (balanced) standard proper coloring.
\end{lemma}
\begin{proof}    Every $g\in G$ acts as an order-preserving map on the poset $X$. Thus, it sends order ideals to order ideals. In particular, the image of a chain $\omega\colon J_1\subsetneq \ldots \subsetneq J_k$ in $J(X)$ is again a chain $g\omega\colon g(J_1)\subsetneq \ldots\subsetneq g(J_k)$. In particular, for all $x,y\in X$ we have $g\omega(x) \leq g\omega(y) \Leftrightarrow \omega(g^{-1}x) \leq \omega(g^{-1}y)$. As noted in \Cref{rem:induced action}, the group $G$ acts also on $\mathbb R^X$ sending $f\in \mathbb R^X$ to $gf:= f\circ g^{-1}$. It follows that $g\Delta_{\omega} = \Delta_{g\omega}$, thus $\phi(g\omega) = g\phi(\omega)$ for all $g\in G$. In particular, the $G$-action preserves the standard triangulation $\triang{X}$.
    
    Now it is enough to prove that the $G$-action on $\mathbb{R}^{X}$ preserves the standard proper coloring on $\triang{X}$. Every vertex of $\triang{X}$ is by definition of the form $v=\Delta_{\omega}$, where $\omega\colon J_1$ is a single-element chain. The color of such a vertex is the cardinality $\vert J_1\vert$. Now for every $g\in G$ the image of the vertex $v$ is $gv=g\Delta_\omega=\Delta_{g\omega}$, where $g\omega $ is the single-element chain $ g(J_1)$. Since $g$ acts by automorphisms of $X$, one has that $\vert g(J_1) \vert = \vert J_1 \vert$, and thus $v$ and $gv$ have the same color.
\end{proof}

We are led to consider the action of $G$ on the poset $J(X)$, and we take the occasion to connect with the setup of Stanley's seminal work \cite{StanleyActions}. 
 
\begin{remark}\label{rem_St1}
Consider a ranked poset $Q$ with a unique minimal element $\hat{0}$ and unique maximal element $\hat{1}$, and write $r$ for the rank of $\hat{1}$. Any action of a group $G$ by automorphisms on $Q$ induces an action on the order complex $\Delta({Q})$. The assignment $\gamma(q)=\operatorname{rk}(q)$ defines a proper and balanced coloring of the vertices of $\Delta({Q})$. Since automorphisms of a ranked poset preserve its rank function, the action on $\Delta({Q})$ preserves this coloring. 

With this setup in mind, for $S\subseteq \{1,\ldots,r-1\}$ Stanley defines 
$$
\alpha_S^Q := \fchar_S^{\Delta(Q)}\quad\quad \beta_S^{Q} := \homchar_{S}^{\Delta({Q})} = \hchar_S^{\Delta({Q})}
$$
where the last equality holds by \Cref{lem:baclawskibjorner} (which, for order complexes of posets, is \cite[Theorem 1.1]{StanleyActions}).
\end{remark}

\begin{remark}\label{rem_St2}
In the setting of \Cref{rem_St1}, note that if $S\subseteq \{0,1,\ldots,r\}$ contains at least one of $0$ and $r$, then $\Delta({Q})\vert_S$ contains at least one of the cone points $\hat{0}$ and $\hat{1}$ and is hence contractible. This implies that its reduced homology vanishes and, thus, $\hchar_S^{\Delta(Q)} = \homchar_S^{\Delta(Q)} = 0$.
\end{remark}

\newcommand{\boh}[2]{\hchar_{#1}^{\Delta(#2)}}

\begin{theorem} \label{thm:orderp}
Let $G$ be a group of automorphisms of a finite poset $X$ and let $\rho$ be the action on $\mathbb{C}[\Delta(J(X))]$ induced by the action of $G$ on $\Delta(J(X))$.

The order polytope $\orp(X)$ is invariant with respect to the permutation representation $\rho^*$ of $G$ on $\mathbb R^X$ given on every $f\in\mathbb R^X$ by $\rho^*(g)(f)=f\circ{g^{-1}}$. The associated equivariant Ehrhart series is
$$\eqehr{\rho^*}{\orp(X)}{t} = \eqhilb{\rho}{\mathbb{C}[\Delta(J(X))]}{t}
$$
and all entries of the equivariant $h$-vector and flag $h$-vector are effective. More precisely,
\begin{equation} \label{eq:ehrhart for order polytopes}
\eqehr{\rho^*}{\orp(X)}{t} = \frac{1}{(1-t)^{d+1}}\sum_{S\subseteq [d-1]} \boh{S}{J(X)}t^{\vert S \vert}
\end{equation}
where $d=\vert X \vert$ and $\boh{S}{J(X)}$ equals the character of the $G$-representation induced on $\widetilde{H}_\ast(\Delta({J(X))\vert_S},\mathbb C)\simeq \widetilde{H}_{\vert S \vert - 1}(\Delta({J(X))\vert_S},\mathbb C) $. Moreover, the equivariant flag $h$-vector refines the equivariant $h$-vector, i.e.,~for all $i\geq 0$ one has that $\boh{i}{J(X)} = \sum_{\vert S \vert = i}\boh{S}{J(X)}$.
\end{theorem}

\begin{proof}
By \Cref{lem:translative_order_polytope}, the action $\rho^*$ on $\triang{X}$ is translative and preserves a balanced coloring, while $\rho$ describes the action on $\Delta(J(X))$ induced by the equivariant isomorphism of abstract simplicial complexes $\phi \colon \usc{\triang{X}}\simeq \Delta(J(X))$. To compute the equivariant Ehrhart series of $\orp(X)$ with respect to $\rho^*$, by \Cref{thm:EBM} it is enough to compute the equivariant Hilbert series of the abstract simplicial complex $\usc{\triang{X}}$ with respect to $\rho^*$. By the equivariant isomorphism of abstract simplicial complexes $\phi \colon \usc{\triang{X}}\simeq \Delta(J(X))$ from \Cref{lem:translative_order_polytope}, this is in turn equal to the equivariant Hilbert series of the abstract simplicial complex $\Delta(J(X))$ with respect to $\rho$. The poset $J(X)$ is bounded with $\hat{0}=\emptyset$ and $\hat{1}=X$, and it is ranked by cardinality of its elements. By \Cref{rem_St2}, we can restrict to color sets $S\subseteq[d-1]$. The claim about the equivariant flag $h$-vector refining the equivariant $h$-vector follows from \Cref{prop:hilbfaces}.(ii).
\end{proof}

\def\Aut{\operatorname{Aut}}

Following standard notation \cite{StanleyActions}, for a ranked poset $P$ of length $\ell$ and $S\subseteq \{1,\ldots, \ell\}$ we write $P_S$ for the poset induced on the set of all elements of $P$ whose rank is in $S$. This implies that $\Delta(P_S)=\Delta(P)\vert_S$.

\begin{example}\label{ex_boolean} 
Let $\mathsf{Anti}_d$ denote the poset consisting of a single antichain with $d$ elements. Then $\orp(\mathsf{Anti}_d)$ is the unit cube $[0,1]^d$.
 Consider the group $G=S_d$, the symmetric group on $d$ elements, acting on $\mathsf{Anti}_d$. The poset $J(\mathsf{Anti}_d)$ is the boolean poset $B_d$ of all subsets of $[d]$, and the induced action of $S_d$ is the natural action permuting elements of $[d]$. As in \Cref{thm:orderp}, let $\rho$ denote the corresponding action on the Stanley-Reisner ring of $\Delta(J(\mathsf{Anti}_d)) = \Delta(B_d)$ and let $\rho^*$ denote the representation of $S_d$ permuting coordinates in $\mathbb R^{\mathsf{Anti}_d}$. Then $\rho^*$ preserves $\orp(\mathsf{Anti}_d)=[0,1]^d$.
 
 The equivariant Hilbert series of $\Delta(B_d)$ with respect to $\rho$ has been computed by Adams and Reiner \cite[Example 2.8]{AdamsReiner} and, with this, we can give an alternative expression for the $S_d$-equivariant Ehrhart series of $\orp(\mathsf{Anti}_d)$ that was first computed in \cite[\S 9]{Stapledon} (see also \cite{StembridgeWeyl}):
 \begin{equation}\label{boolean}
\eqehr{\rho^*}{\orp{(\mathsf{Anti}_d)}}{t} = \frac{\sum_{\lambda} \chi_\lambda  t^{\operatorname{des}(\lambda)}}{(1-t)^{d+1}},
\end{equation}
where $\lambda$ runs across all standard Young tableaux of size $d$, $\operatorname{des}(\lambda)$ is the cardinality of the descent set of the tableau $\lambda$, and $\chi_\lambda$ is the character of the $S_d$-representation associated with the shape of $\lambda$. 
\end{example}

\begin{example}\label{ex:radiotower}
For $k>0$ let $\mathsf{Rad}_k$ be the ``radio-tower'' poset with $2k$ elements as in \Cref{RTP}, where the poset $J(\mathsf{Rad}_k)$ is also depicted. We set $A_i := (\mathsf{Rad}_k)_{\leq a_i}$, $B_i := (\mathsf{Rad}_k)_{\leq b_i}$ for all $i\geq 0$, $C_i := A_{i}\cup B_{i}$ for all $i>0$, and $C_0:=\emptyset$. On $\mathsf{Rad}_k$ we consider the action of the group $(S_2)^k$, where the $i$-th factor swaps $a_i$ and $b_i$. The induced action on $J(\mathsf{Rad}_i)$ fixes the elements $C_i$ and switches $A_i$ with $B_i$.

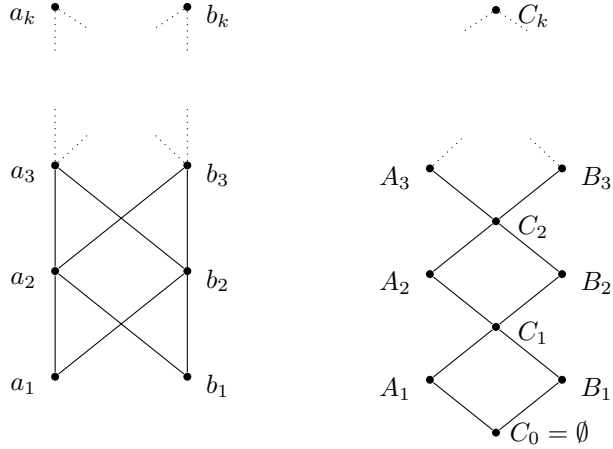
\begin{figure}[h!]
\begin{tikzpicture}
\node[anchor=north] (A) at (0,0) {
\begin{tikzpicture}[x=5em,y=2em]
\foreach \x in {0,2,4,7}
    \node[circle,fill=black,inner sep=0pt,minimum size=3pt] (a\x) at (0,\x) {};
\node[label=left:{$a_1$}] at (a0) {};
\node[label=left:{$a_2$}] at (a2) {};
\node[label=left:{$a_3$}] at (a4) {};
\node[label=left:{$a_{k}$}] at (a7) {};    
\foreach \x in {0,2,4,7}    
    \node[circle,fill=black,inner sep=0pt,minimum size=3pt,] (b\x) at (1,\x) {};
\node[label=right:{$b_1$}] at (b0) {};
\node[label=right:{$b_2$}] at (b2) {};
\node[label=right:{$b_3$}] at (b4) {};
\node[label=right:{$b_{k}$}] at (b7) {};
\draw (a2) -- (a4) -- (b2) -- (a0) -- (a2) -- (b4) -- (b0) -- (a2);
\draw[dotted] (a4) -- (0,5);
\draw[dotted] (b4) -- (1,5);
\draw[dotted] (a7) -- (0,6);
\draw[dotted] (b7) -- (1,6);
\draw[dotted] (a4) -- (0.25,4.5);
\draw[dotted] (b4) -- (0.75,4.5);
\draw[dotted] (a7) -- (0.25,6.5);
\draw[dotted] (b7) -- (0.75,6.5);
\node (bottom) at (0,-2) {};
\end{tikzpicture}};
\node[anchor=north] (B) at (5,0) {
\begin{tikzpicture}[x=5em,y=2em]
\foreach \x in {0,2,4}
    \node[circle,fill=black,inner sep=0pt,minimum size=3pt] (a\x) at (0,\x) {};
\node[label=left:{$A_1$}] at (a0) {};
\node[label=left:{$A_2$}] at (a2) {};
\node[label=left:{$A_3$}] at (a4) {};
\foreach \x in {0,2,4}    
    \node[circle,fill=black,inner sep=0pt,minimum size=3pt,] (b\x) at (1,\x) {};
\node[label=right:{$B_1$}] at (b0) {};
\node[label=right:{$B_2$}] at (b2) {};
\node[label=right:{$B_3$}] at (b4) {};
\foreach \x in {1,3,7}    
    \node[circle,fill=black,inner sep=0pt,minimum size=3pt,] (c\x) at (.5,\x) {};
\node[label=right:{$C_1$}] at (c1.north east) {};
\node[label=right:{$C_2$}] at (c3.north east) {};
\node[label=right:{$C_{k}$}] at (c7.north east) {};
\node[circle,fill=black,inner sep=0pt,minimum size=3pt,label=right:{$C_0=\emptyset$}] (cb) at (.5,-1) {};
\draw (cb) -- (a0) -- (c1) -- (a2) -- (c3) -- (a4);
\draw (cb) -- (b0) -- (c1) -- (b2) -- (c3) -- (b4);
\draw[dotted] (a4) -- (0.25,4.5);
\draw[dotted] (b4) -- (0.75,4.5);
\draw[dotted] (c7) -- (0.25,6.5);
\draw[dotted] (c7) -- (0.75,6.5);
\end{tikzpicture}};
\end{tikzpicture}
\caption{The ``radio-tower poset'' $\mathsf{Rad}_k$ (left-hand side) and its poset of ideals $J(\mathsf{Rad}_k)$ (right-hand side).} \label{RTP}
\end{figure}

Note that $J(\mathsf{Rad}_k)$ is ranked, the rank of $C_i$ equals $2i$ and the rank of $A_i$ and $B_i$ is $2i-1$. In order to compute the equivariant Ehrhart series of the order polytope associated with $\mathsf{Rad}_k$, we consider $S\subseteq\{1,\ldots,2k-1\}$ and we study the $G$-representation on the homology of $J(\mathsf{Rad}_k)_S$. If $S$ contains an even number, then $J(\mathsf{Rad}_k)_S$ contains some $C_i$ and is contractible. In particular, its top reduced homology is trivial and affords the zero character. 
Otherwise, $J(\mathsf{Rad}_k)_S$ is isomorphic to $\mathsf{Rad}_{\vert S \vert}$ and the associated order complex is, topologically, an $(\vert S \vert -1)$-dimensional sphere that can be realized as the boundary complex of the cross-polytope whose vertices are the standard basis vectors of $\mathbb R^{\vert S \vert}$ and their negatives. The action of a generator of $G$ corresponds to reflecting the cross-polytope about a coordinate hyperplane, and hence it reverses orientation. 
This defines an (irreducible) representation of $G$ on the (one-dimensional) top homology of $J(\mathsf{Rad}_k)_S$, isomorphic to the representation in which  $(\tau_1,\ldots,\tau_k)\in (S_2)^k$ acts on $\mathbb C$ by multiplication with $\sgn(\tau_{1})\cdots \sgn(\tau_{k})$. Call $\sgn_j$ the sign representation of the $j$-th factor in $(S_2)^k$. Then,  $\beta_S^{J(\mathsf{Rad}_k)} = \prod_{i\in S}\sgn_{\frac{i+1}{2}}$ in $\basering{G}$. 
We obtain the following expression for the equivariant Ehrhart series of $\orp(\mathsf{Rad}_k)$ with respect to the permutation action on the coordinates of $\mathbb R^{\mathsf{Rad}_k}$ (note that, since every element of this group is an involution, this is the ``usual'' permutation representation): 
$$\eqehr{\rho^*}{\orp(\mathsf{Rad}_k)}{t} = \frac{1}{(1-t)^{2k+1}}\sum_{S\subseteq [2k]_{\textrm{odd}}} t^{\vert S \vert}\prod_{i\in S}\sgn_{\frac{i+1}{2}} 
= \frac{\prod_{j=1}^{k}(1+\sgn_j t)}{(1-t)^{2k+1}}.
$$
\end{example}

\Cref{ex:radiotower} is a special case of the following general fact. Following \cite[\S 9.4]{TopologicalMethods}, the ordinal sum $P\ast Q$ of two posets is the poset on the disjoint union $P\sqcup Q$ with order relation $x\leq y$ if $x,y\in P$ and $x\leq_P y$, or $x,y\in Q$ and $x\leq_Q y$, or $x\in P, y\in Q$. In particular, the $k$-th radio-tower poset $\mathsf{Rad}_k$ is the $k$-fold ordinal sum of the $2$-element antichain. If $\Sigma_1$ and $\Sigma_2$ are two abstract simplicial complexes with no vertices in common, we denote by $\Sigma_1 \ast \Sigma_2$ their join, i.e., the abstract simplicial complex whose faces are the collections $\sigma_1 \cup \sigma_2$ with $\sigma_1 \in \Sigma_1$ and $\sigma_2 \in \Sigma_2$.

\begin{proposition} \label{prop:joinorder}
Let $X_1,X_2$ be two finite posets and consider the action  of a finite group $G$ on the ordinal sum $X_1\ast X_2$. Then $G$ must preserve $X_1$ and $X_2$, and 
$$
\eqehr{\rho^*}{\orp(X_1\ast X_2)}{t} =
(1-t)\eqehr{\rho^*_1}{\orp(X_1)}{t} 
\eqehr{\rho^*_2}{\orp(X_2)}{t} 
$$
where $\rho^*$, $\rho^*_1$, $\rho^*_2$ are the permutation representations of $G$ on $\mathbb C^X$, $\mathbb C^{X_1}$, $\mathbb C^{X_2}$ defined as in \Cref{thm:orderp}. 
\end{proposition}

\def\djxxs{\mathbf{J}_{12}}
\newcommand{\djxs}[1]{\mathbf{J}_{#1}}

\begin{proof} 
Note first that $\mathrm{Aut}(X_1 \ast X_2) \cong \mathrm{Aut}(X_1) \times \mathrm{Aut}(X_2)$, and the given action on $X_1 \ast X_2$ gives rise by restriction to actions on $X_1$ and on $X_2$. Let $r_i:=\vert X_i\vert$ for $i\in \{1,2\}$. Given $S\subseteq\{1,\ldots,r_1+r_2-1\}$, write $S_1:=S_{\leq r_1}$, $S_2:=\{s-r_1\mid s\in S_{>r_1}\}$. For every such $S$ one checks that 
$$
J(X_1\ast X_2)_S =
J(X_1)_{S_1}\ast J(X_2)_{S_2}.
$$
For brevity we write $\djxxs:= \Delta(J(X_1\ast X_2)_S) = \Delta(J(X_1\ast X_2))\vert_S$ and $\djxs{i}:= \Delta(J(X_{i})_{S_{i}}) = \Delta(J(X_i))\vert_{S_i}$ for $i\in\{1,2\}$. Since the order complex of an ordinal sum of posets is the join of the order complexes of those posets (see \cite[\S 9.4]{TopologicalMethods}), we have
\begin{equation} \label{eq:joins}
\djxxs = \djxs{1}\ast \djxs{2}.
\end{equation}
By \cite[Ch. 5, \S 2, Corollary 2.3]{CookeFinney}, for every $k$ there is an isomorphism

$$\widetilde{H}_k(\djxs{1}\ast\djxs{2}) \to \bigoplus_{i+j=k-1} \widetilde{H}_i(\djxs{1})\otimes \widetilde{H}_j(\djxs{2})$$
(recall that we use $\mathbb C$-coefficients -- although here this holds also for integer homology, as the homologies of the involved complexes are torsion-free, since posets of order ideals are distributive lattices and thus shellable \cite{Provan-decompositions}).

By \cite[Theorem 4.3]{StanleyBalanced}, the color-selected subcomplexes $\djxxs$, $\djxs{1}$ and $\djxs{2}$ are Cohen--Macaulay, hence their reduced homology is either zero or concentrated in top degree. Thus, the only nontrivial induced isomorphism in homology is
$$
\widetilde{H}_{\vert S \vert-1} 
(\djxs{1}\ast \djxs{2}) \to 
\widetilde{H}_{\vert S_1\vert-1}(\djxs{1})\otimes 
\widetilde{H}_{\vert S_2\vert -1}(\djxs{2}).
$$
Now, this isomorphism as constructed in \cite[Chapter 5, \S 1 and \S 2]{CookeFinney} is induced by a chain map that sends a chain of $\djxs{1}\ast \djxs{2}$ to the tensor product of the two associated chains of $\djxs{1}$ and $\djxs{2}$ and is thus equivariant. Thus we have $\homchar_S^{\Delta(J(X_1\ast X_2))}=
\homchar_{S_1}^{\Delta(J(X_1))}
\homchar_{S_2}^{\Delta(J(X_2))}$ and so
$$
\boh{S}{J(X_1\ast X_2)} =
\boh{S_1}{J(X_1)}
\boh{S_2}{J(X_2)}
$$
Now notice that the denominator of the rational expression on the right-hand side of \eqref{eq:ehrhart for order polytopes} for the equivariant Ehrhart series of $\orp(X_1\ast X_2)$ has degree $r_1+r_2+1$, while the degree of the analogous denominator for $\orp(X_i)$ is $r_i+1$, for $i\in\{1,2\}$. We obtain the claimed equality.

\end{proof}

\section{Alcoved polytopes} \label{sec:alcoved}

Let $W$ be an irreducible Weyl group of rank $d$ with associated root system $\Phi_W\subseteq \mathbb R^d$ and nondegenerate symmetric bilinear form $(\cdot ,\cdot)$ on $\mathbb R^d$. We will follow \cite{Bourbaki,Humphreys} as standard references about Coxeter and Weyl groups. 

Let $\AW$ be the reflection arrangement of the corresponding affine Coxeter group $\tW$. This means that $\AW$ contains all hyperplanes $H_k^\alpha:=\{x \mid (x,\alpha)=k\}\subseteq \mathbb R^d$ where $\alpha \in \Phi_W$ and $k\in \mathbb Z$. The planes of $\AW$ define a triangulation $\TT_W$ of the Euclidean space $\mathbb R^d$ into simplices (this is sometimes called the {\em Coxeter complex} of $\tW$). This simplicial complex is necessarily pure. 
An {\em alcove} of type $W$ is the relative interior of any maximal simplex of $\TT_W$, see  \cite[\S 4.3]{Humphreys}, \cite[VI.2.1]{Bourbaki}. 

\begin{lemma} The action of $\tW$ on $\TT_W$ preserves a balanced proper coloring. In particular, it is translative.
\end{lemma}
\begin{proof}
Pick any maximal simplex $\Delta$ of $\TT_W$. Then $\Delta$ is the closure of an alcove. It  is well-known that $\Delta$ is a fundamental region for the action of $\tW$ on $\mathbb R^d$ (see, e.g., \cite[\S 4.8]{Humphreys} and \cite[VI.2.1]{Bourbaki}). 
 In particular, every vertex $v$ of the triangulation $\TT_W$ is in the orbit of a unique vertex $v_\Delta$ of $\Delta$. Moreover, the group $\tW$ acts transitively on the set of alcoves \cite[\S 4.3]{Humphreys}, and purity of $\TT_W$ implies that every simplex of $\TT_W$ is in the closure of some alcove. Thus, for any maximal simplex $\sigma$ of $\TT_W$, there is a unique $w\in \tW$ with $w\sigma=\Delta$, and hence no two vertices of $\sigma$ are in the same orbit. 
Since $\Delta$ has $d +1$ vertices, the assignment $v\mapsto v_\Delta$ gives a balanced ``coloring'' of all vertices of the arrangement that is clearly preserved by the group action. Thus the action on the vertices of the arrangement preserves a balanced proper coloring.
\end{proof}

\begin{definition}[Alcoved polytopes {\cite{LaPo2}}] \label{def:Walcoved}
A polytope $P\subseteq \mathbb R^d$ is {\em alcoved} (of type $W$) if $P$ is a union of cells of $\TT_W$. 
Then, $\TT_W$ restricts to a triangulation $\TT_P$ of $P$. 
\end{definition}

\begin{corollary}\label{cor:balalc}
Let $P$ be an alcoved polytope of type $W$ and let $G$ be a subgroup of $\tW$ that preserves $P$. Then $G$ acts on $\TT_P$ preserving a balanced proper coloring.
\end{corollary}

Choose a maximal simplex $\Delta_0$ of $\TT_W$ containing the origin. Choose $\alpha_1,\ldots,\alpha_d\in \Phi_W$ such that $H_0^{\alpha_i}$ are walls of $\Delta_0$  and $(\alpha_i,\Delta_0)\geq 0$.  Then $\alpha_1,\ldots,\alpha_d$ is a valid choice for a system of {\em simple roots} of $W$ and is a basis of $\mathbb R^d$ (\cite[\S 4.3]{Humphreys}, \cite[VI.2.1]{Bourbaki}). Define $\omega_1,\ldots,\omega_d$ via $(\omega_i,\alpha_j)=\delta_{ij}$. The $\omega_i$ are the {\em fundamental coweights} of $W$. Since $W$ is irreducible, there is a unique root $\theta$ of maximal height\footnote{The height of $\alpha\in \Phi_W$ is $(\alpha, \omega_1+\ldots+\omega_d)$.}. Let $\overline{\omega}_i:=\omega_i/(\theta,\omega_i)$. It is known (\cite[\S 4.9]{Humphreys}, \cite[VI.2.2, Corollary]{Bourbaki}) 
that 
$$
\Delta_0 = \operatorname{conv} \{0,\overline{\omega_1},\ldots,\overline{\omega_d}\}. 
$$

\begin{definition}
Let $M_W$ denote the lattice in $\mathbb R^d$ generated by $\overline{\omega_1},\ldots,\overline{\omega_d}$.
\end{definition}

The lattice $M_W$ is the {\em coweight lattice} of $W$. 
The action of $W$ on $\mathbb R^d$ restricts to an action on $M_W$ by unimodular linear transformations and the action of $\tW$ restricts to an action on $M_W$ by affine automorphisms 
(see \cite[\S 2.9.(2) and \S 4.9]{Humphreys} and \cite[VI.2.2]{Bourbaki}).

\begin{lemma}\label{lem:AlUnTri}
The triangulation $\TT_W$ is unimodular with respect to the lattice $M_W$.
\end{lemma}

\begin{proof} $\Delta_0$ is unimodular by definition. All other maximal simplices of $\TT_W$ are the image of $\Delta_0$ by some element $g$ of $\tW$. Since $g$ is an affine automorphism, the fact that $\{\overline{\omega_1},\ldots,\overline{\omega_d}\}$ is a lattice basis of $M_W$ implies that $\{g(0)-g(\overline{\omega_1}),\ldots, g(0)-g(\overline{\omega_d})\}$ is a lattice basis as well. Therefore $g(\Delta_0)$ is unimodular. 
\end{proof}

\begin{theorem}\label{thm:Walcove}
    Let $P$ be an alcoved polytope of type $W$ and let $G$ be a subgroup of $\tW$ preserving $P$. Then $P$ is an $M_W$-lattice polytope and the equivariant Ehrhart series of $P$ equals the equivariant Hilbert series of the induced $G$-action on the simplicial complex $\TT_P$. Morever, all entries of the equivariant $h$-vector and equivariant flag $h$-vector are effective.
\end{theorem}

\begin{proof}
    The claim follows from an application of \Cref{thm:EBM}, whose hypotheses are satisfied by \Cref{cor:balalc} and \Cref{lem:AlUnTri}.
\end{proof}

\def\bV{\epsilon}

\begin{remark}\label{twoalcoves}
If the Weyl group $W$ is not specified, type $A$ is generally understood (e.g., as in \cite{LaPo}). In this case, say for type $A_{d+1}$, the  construction of the arrangement $\AWA{d+1}=\{H_k^\alpha\}_{k\in \mathbb Z, \alpha\in \Phi_W}$ described at the beginning of the section is an arrangement in the linear subspace $V\subseteq \mathbb R^{d+1}$ spanned by the vectors $\alpha_i:=\bV_i-\bV_{i+1}$, $1\leq i\leq d$ (we use $\bV_{i}$ for the standard basis vectors of $\mathbb R^{d+1}$ for better clarity), and the bilinear form $(\cdot \mid \cdot)$ is defined for every $v,w\in V$ as $(v\mid w)=\frac{v^Tw}{d+1}$ \cite[Plate 1, p.~265]{Bourbaki}. 

In the context of alcoved polytopes (of type $A$) it is customary \cite{LaPo} to rather consider triangulations of $\mathbb R^{d}$ defined by the arrangement  $\mathscr A_d$ with hyperplanes
\begin{equation}\label{eq:Ad}
\begin{array}{rl}
H^{i,j}_k &=\{x_i-x_j = k \mid k\in \mathbb Z\},
\textrm{ for }1\leq i<j\leq d\\[.5em]
H^{i}_k &=\{x_i = k \mid k\in \mathbb Z\}, \textrm{ for }1\leq i\leq d
\end{array}
\end{equation}
This is justified, as $\mathscr A_d$ is the image of $\AWA{d+1}$ under the nonsingular linear transformation $f: V \to \mathbb R^d$ sending the basis $\alpha_1,\ldots,\alpha_d$ of $V$ to the basis $\beta_1,\ldots,\beta_d$ of $\mathbb R^d$ given by
$\beta_i = ({e_i-e_{i+1}})({{d+1}})$ for $1\leq i \leq d-1$  and $\beta_d={(e_1+e_2+\ldots + e_d)}({{d+1}})$  (we use $e_i$ for the standard basis of $\mathbb R^d$, the image of $f$).

Indeed, for any root $\alpha = \bV_{i}-\bV_{j}$ of $A_{d+1}$ (say $i<j$), the hyperplane $H_{k}^\alpha$ consists of all vectors $y_1\alpha_1 + \ldots + y_d\alpha_d \in V$ satisfying one of the following conditions, depending on the values of $i,j$.
$$
\begin{array}{l|l}
\alpha=\epsilon_i-\epsilon_j & \textrm{ equation of } H_k^{\alpha} \textrm{ in the basis }(\alpha_l)_l\\\hline
1< i < j < d+1 &
(y_i - y_{i-1}) - (y_j-y_{j-1}) = k(d+1)\\
i=1,\, j<d+1 & 
y_1 - (y_j-y_{j-1})= k(d+1) \\
i>1,\,j=d+1 &
(y_i-y_{i-1}) + y_d=k(d+1)\\
i=1,\, j=d+1 & y_1+y_d = k(d+1)
\end{array}
$$
Since a vector with coordinates  $(y_1,\ldots,y_d)$ with respect to the basis $\beta$ has coordinates  $$({d+1})(y_1+y_d,(y_2-y_1)+ y_d, (y_3-y_2)+y_d,\ldots, (y_{d-1}-y_{d-2})+y_d,y_{d} - y_{d-1})$$ with respect to the standard basis of $\mathbb R^{d}$, then $f(H_{k}^{\alpha})$ is the set of points in $\mathbb R^{d}$ of the form $x_1e_1+\ldots +x_de_d$ with

$$
\begin{array}{l|l}
(i,j) & \textrm{ equation of } H_k^{(i,j)} \textrm{ in the basis }(e_l)_l\\\hline
1< i < j < d+1 & x_i-x_j = k \\
i=1,\, j<d+1 & x_1-x_j = k\\
i\geq 1, j=d+1 & x_i =k 
\end{array}
$$

In particular, $f(M_W)=\mathbb Z^d$, thus $f$ defines an isomorphism of lattices $M_W\to \mathbb Z^d$ that maps any $A_{d+1}$-alcoved polytope $P$ in the sense of \cite{LaPo2} (see \Cref{def:Walcoved}) into a $\mathbb Z^d$-lattice polytope $f(P)\subseteq \mathbb R^d$ that is alcoved in the sense of \cite{LaPo}, where alcoves are defined as the chambers of the arrangement $\AA_{d}$ described in \Cref{eq:Ad}. Unimodularity of $\TT_W$ with respect to $M_W$ then implies $\mathbb Z^d$-unimodularity of the triangulation by chambers of $\AA_{d}$.
\end{remark}

\begin{theorem} \label{prop:alcoved}
Let $P\subseteq \mathbb R^d$ be a polytope that is alcoved by the hyperplanes in \eqref{eq:Ad} and let $\TT$ the resulting triangulation of $P$. Let a group $G$ act on $\mathbb R^d$ by permuting the coordinates. If $P$ is preserved by the action, then the induced action on $\TT$ is translative and preserves a balanced proper coloring. In particular, the equivariant Ehrhart series of $P$ equals the equivariant Hilbert series of the induced $G$-action on $\TT$ and all entries of the equivariant $h$-vector and equivariant flag $h$-vector are effective.
\end{theorem}

\begin{proof} Recall the setup of \Cref{twoalcoves} and \cite[Plate 1, p.~265]{Bourbaki}. The vectors $\alpha_1,\ldots,\alpha_d$ are a system of simple roots of a Coxeter system of type $A_{d}$, numbered so that $\alpha_i$ is adjacent to $\alpha_{i+1}$ in the Dynkin diagram of $A_d$, for all $0<i<d$. Therefore $\alpha_1,\ldots,\alpha_{d-1}$ is the system of simple roots of a Weyl group of type $A_{d-1}$ generating a standard parabolic subgroup of the original system. Notice that $f$ maps  $\alpha_1,\ldots,\alpha_{d-1}$ to $\beta_1,\ldots,\beta_{d-1}$, and the vectors $\{\beta_i/(d+1)\}_{i=1,\ldots,d-1}$ are a standard choice of simple roots for the representation of $A_{d-1}$ in $\mathbb R^d$ permuting the coordinates. In particular, the action of the symmetric group permuting coordinates in $\mathbb R^d$ corresponds via the map $f$ to the action of a parabolic subgroup of $\widetilde{A}_{d}$ on the chambers of $\AWA{d}$ and is, thus, translative on the chambers of $\AA_{d}$.
\end{proof}

\begin{remark}\label{rem:ORPA} The order polytopes considered in \Cref{sec:OP} are alcoved by the hyperplanes in \Cref{eq:Ad}. In this sense, the effectiveness claim of \Cref{thm:orderp}
can be obtained as a consequence of \Cref{prop:alcoved}. 
\end{remark}

We close our discussion by mentioning a special class of alcoved polytopes of type $A$ recently introduced by Sanyal and Stump.

\begin{definition}[{\cite[\S 1]{SaSt}}]\label{def:lip}
The {\em Lipschitz polytope} of a finite poset $X$ is 
$$
\Lip(X)= \{f\in \mathbb R^{X} \mid 0\leq f(x) \leq 1 \textrm{ for }x\in \min X;\ 
0\leq f(y) - f(x) \leq 1 \textrm{ for }x\lessdot y\}
$$
(where the notation ``$x \lessdot y$'' means that $y$ covers $x$, i.e.: $x < y$ and there is no $z \in X$ such that $x < z < y$).

The polytope $\Lip(X)$ is full-dimensional in $\mathbb R^X$ and it is alcoved by the hyperplanes of the arrangement $\AA_d$ described in \Cref{eq:Ad}. Call $\TT_X$ the resulting triangulation of $\Lip(X)$.
\end{definition}

The polytopes $\Lip(X)$ are $\mathbb Z^X$-lattice polytopes, and they are alcoved  in the sense of \Cref{twoalcoves}. 
Now let $G$ be a group of poset automorphisms of $X$. Then every $g\in G$ is a permutation of the set $X$ that preserves the order relation (in particular, since $X$ is finite, it permutes the minimal elements of $X$ and preserves covering relations). 

\begin{remark}
Let $\rho^*$ be the representation of $G$ on $\mathbb R^X$ given by $\rho^*(g)(f)=f\circ g^{-1}$ as in \Cref{rem:induced action}.(i). A straightforward check of \Cref{def:lip} shows that $\Lip(gX)=\rho^*(g)\Lip(X)$, hence $\Lip(X)$ is invariant under the $G$-representation $\rho^*$.
\end{remark}
  
  We obtain the following corollary to \Cref{thm:EBM} via 
 \Cref{prop:alcoved}.

\begin{corollary}
Let $\rho \colon G \to \mathrm{Aut}(X)$ be a group homomorphism defining an action of $G$ on a finite poset $X$ via poset automorphisms.
Then $G$ acts on the $\mathbb Z^X$-lattice polytope $\Lip(X)$ via the (permutation) representation $\rho^*$. 
The equivariant Ehrhart series of $\Lip(X)$ with respect to $\rho^*$ equals the equivariant Hilbert series of the induced $G$-action on the simplicial complex $\TT_X$. Morever, all entries of the equivariant $h$-vector and equivariant flag $h$-vector are effective.
\end{corollary}
\begin{proof} 

As discussed above, the representation $\rho^*$ preserves the polytope $\Lip(X)$. Since $\rho^*$ permutes the coordinates of $\mathbb R^X$, by \Cref{prop:alcoved} it also acts on the triangulation $\TT_X$, and preserves a balanced coloring. 
It is well-known that $\TT_X$ is unimodular (see also the end of \Cref{twoalcoves}), thus \Cref{thm:EBM} applies.
\end{proof}

\begin{example}\label{ex:lips}
As an example we compute the equivariant Ehrhart series of the Lipschitz polytope of \cite[Example 2.6]{SaSt}. Consider, for a fixed $n>0$, the poset $T_n$ with $n$ elements of rank $0$ (labeled $1$ through $n$) and one element of rank one (that we label by $n+1$) that is comparable with every other element. Note that this labeling is natural in the usual poset-theoretic sense.

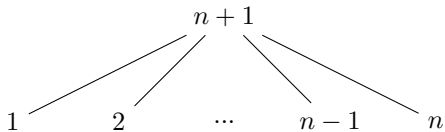
\begin{figure}[h!]
\begin{center} 
\begin{tikzpicture}[x=2em, y=4em]
\node (up) at (0,1) {$n+1$};
\node[anchor=center] (1) at (-4,0) {$1$};
\node[anchor=center] (2) at (-2,0) {$2$};
\node[anchor=center] (pts) at (0,0) {...};
\node[anchor=center] (nu) at (2,0) {$n-1$};
\node[anchor=center] (n) at (4,0) {$n$};
\draw (up) -- (1);
\draw (up) -- (2);
\draw (up) -- (nu);
\draw (up) -- (n);
\end{tikzpicture}
\end{center}
\caption{The poset $T_n$.}
\end{figure}

Let $G$ be the group of automorphisms of $T_n$, i.e., $G$ is the symmetric group $S_n$ acting by permutation of the elements $1,\ldots,n$ of rank $0$. Sanyal and Stump describe the maximal simplices of the unimodular triangulation by alcoves as follows. Let $\tau\in S_{n+1}$ be any permutation of the set of elements of $T_n$. We call $\tau$ \defil{descent-compatible} if, for every chain $a\lessdot b$ in $T_n$, the number of descents of the word $\tau(a)\tau(b)$ only depends on $b$: we call $d_{\tau}(b)$ this number. In our situation this means that either (1) $\tau(n+1)>\tau(i)$ for all $i\leq n$ or (2) $\tau(n+1)<\tau(i)$ for all $i\leq n$. We conclude that $\tau$ is descent-compatible if and only it is of one of two types.
\begin{itemize}
\item Type 1: $\tau(n+1)=n+1$. Here $d_\tau(i) = 0$ for all $i\in [n+1]$. 
\item Type 2: $\tau(n+1)=1$. Here $d_\tau (i) = 0$ when $i\leq n$ and $d_\tau(n+1) = 1$.
\end{itemize}

Given a descent-compatible permutation $\tau$, consider a point $q_{\tau}\in \mathbb R^{n+1}=\mathbb R^{T_n}$ defined by $q_i:= d_\tau(i)$. Then 
$$
q_\tau = (0,\ldots,0) \textrm{ if }\tau\textrm{ is of type (1)},\quad
q_\tau = (0,\ldots,0,1) \textrm{ if }\tau\textrm{ is of type (2)}.
$$
By \cite[Theorem 3.1]{SaSt} the family of simplices
\def\tinv{\tau^{-1}}
$$
\sigma_\tau := q_\tau + \{x\in \mathbb R^{n+1} \mid 0\leq x_{\tinv(1)}\leq \ldots \leq x_{\tinv(n+1)} \leq 1\}
$$
is the set of maximal cells of the unimodular triangulation $\TT_{T_n}$ of $\Lip(T_n)$. With what we said before, we have again two types of such simplices, of which we can describe the vertices explicitly (we will set $v_I:=\sum_{i\in I} e_i$ for $I\subseteq [n+1]$).
\begin{itemize}
\item Type 1: $\sigma_\tau = \{x\in \mathbb R^{n+1} \mid 0\leq x_{\tinv(1)}\leq \ldots \leq x_{n+1} \leq 1\} $, with vertices given by all points $v_{\emptyset}, v_{\{n+1\}}, v_{I_2} \ldots v_{I_{n+1}}$ where $\emptyset \subset \{n+1\} \subset I_2 \subset \ldots \subset I_{n+1}=[n+1]$ is a maximal chain in the boolean poset $B_{n+1}$ passing through $\{n+1\}$.
\item Type 2: $\sigma_\tau = e_{n+1} + \{x\in \mathbb R^{n+1} \mid 0\leq x_{n+1}\leq x_{\tinv(2)} \leq \ldots \leq x_{\tinv(n+1)} \leq 1\} $, with vertices given by all points $e_{n+1}+v_{I_0}, \ldots, e_{n+1}+v_{I_n}$ where $\emptyset=I_0\subset \ldots \subset I_{n+1}=[n+1]$ is a maximal chain in the boolean poset $B_{n+1}$ with $n+1\in I_i$ only if $i=n+1$.
\end{itemize}

Notice that every maximal simplex is of the form
$$
\Delta_{I_*,x}:=\conv\{e_{n+1}+v_{I_0}, \ldots, e_{n+1}+v_{I_n}\}\cup \{x\} 
$$
\def\ta{a} \def\tz{z}
where $I_*: \emptyset =I_0 \subsetneq I_1\subsetneq \ldots \subsetneq I_{n}=[n]$ is a maximal chain in $B_n$ and $x$ is either the point $\ta:=(0,\ldots,0)$ or the point $\tz:=v_{[n+1]}+e_{n+1}$.

This simplicial complex is isomorphic to the order complex of the poset $Y$ obtained from $B_n$ by adding two elements $\ta$, $\tz$ that we declare to be incomparable with each other but strictly greater than every element of $B_n$ (see \Cref{figX}).

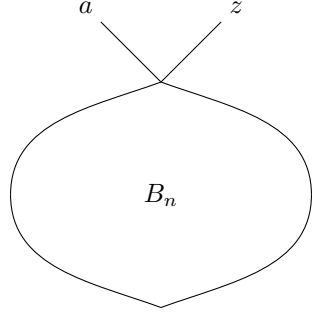
\begin{figure}
\begin{center}
\begin{tikzpicture}
    \node (alpha) at (-1,2.5) {$\ta$};
    \node (omega) at (1,2.5) {$\tz$};   
    \draw (0,1.5) edge[out=200,in=90] (-2,0);
    \draw (-2,0) edge[out=270,in=160] (0,-1.5);
    \draw (0,1.5) edge[out=340,in=90] (2,0);
    \draw (2,0) edge[out=270,in=20] (0,-1.5);
    \node (B) at (0,0) {$B_n$};
    \draw (alpha) -- (0,1.5);
    \draw (omega) -- (0,1.5);    
\end{tikzpicture}
\end{center}
\caption{The poset $Y$ for \Cref{ex:lips}.}\label{figX}
\end{figure}

Now recall the group $G=S_n$ of permutations of $[n]\subseteq T_n$ and its action $\rho^*$ on $\mathbb R^{T_n}$ given by $\rho^*(g)(f)=f\circ g^{-1}$. Notice that $\ta$, $\tz$ as well as $e_{n+1}$ are fixed by every group element. Moreover, for every $I\subseteq [n]$ we have
$
\rho^*(g)(v_I)= v_{g(I)}
$ and so $\rho^*(g)\Delta_{I_*,x}=\Delta_{g(I_*),x}$.

This action corresponds to the natural action of $G$ on $Y$ that fixes $a$ and $z$ and is such that any $g\in G$ sends every $I\subseteq [n]$ to $gI$. Now, for every $g\in G$ the fixed complex $\Delta(Y)^g$ is the suspension (at $\ta$ and $\tz$) of $\Delta(B_n)^g$; therefore, the equivariant Hilbert series of $\Delta(Y)$ evaluated at $g$ is
$$
\frac{1+t}{1-t}\hilb{\mathbb{C}[\Delta(B_n)^g]}{t}.
$$
We obtain
\begin{equation} \label{eq:Lipschitz example}
\eqehr{\rho^*}{\Lip(T_n)}{t}=\frac{1+t}{1-t}\eqhilb{\pi}{\mathbb{C}[\Delta(B_n)]}{t},
\end{equation}
where 
$\pi$ is the canonical permutation action of $S_n$ on $B_n$. The computation of the right-hand side of \eqref{eq:Lipschitz example} can now be completed by recalling \Cref{ex_boolean}. 

\end{example}

\bibliographystyle{alpha} 
\bibliography{bibliography}

\end{document}